\documentclass[11pt,twoside]{article}

\usepackage{mathrsfs,amsfonts,amsmath,amssymb}
\usepackage{latexsym}
\newtheorem{satz}{Theorem}
\newtheorem{theorem}{Theorem}
\newtheorem{lemma}[satz]{Lemma}
\newtheorem{corollary}[satz]{Corollary}
\newtheorem{remark}[satz]{Remark}

\newcommand{\bpm}{\begin{pmatrix}}
\newcommand{\epm}{\end{pmatrix}}

\def\T{\mathsf{T}}

\def\F{\mathbb {F}}
\def\R{\mathbb {R}}
\def\E{\mathsf{E}}
\def\Q{\mathsf{Q}}

\def\C{\mathbb{C}}

\def\d{\delta}

\def\({\big (}
\def\){\big )}

\def\G{\Gamma}

\def\le{\leqslant}
\def\ge{\geqslant}
\def\_phi{\varphi}
\def\eps{\varepsilon}

\def\t{\tilde}

\def\D{\Delta}

\def\tr{{\mathfrak{tr}}}

\def\SL{{\rm SL_2}}
\def\Conj{{\rm Conj}}
\def\Stab{{\rm Stab}}
\def\Aff{{\rm Aff}}

\author{M. Rudnev\footnote{
		The first author is supported by the Leverhulme  Trust Grant RPG--2017--371.} , I. D. Shkredov}
\title{On growth rate in $\SL(\F_p)$, the affine group and  sum-product type implications}
	
\date{}
\begin{document}
	\maketitle

\begin{abstract}
This paper aims to study in more depth the relation between growth in matrix groups $\SL (\F)$  and $\Aff(\F)$  over a field $\F$ by multiplication and geometric incidence estimates, associated with the sum-product phenomenon over $\F$. It presents streamlined proofs of  Helfgott's  theorems on growth in the $\F_p$-case, which avoid sum-product estimates. For $\SL(\F_p)$, for sets exceeding in size some absolute constant, we improve the lower bound $\frac{1}{1512}$ for the growth exponent, due to Kowalski, to $\frac{1}{20}.$  For the affine group we fetch a sharp theorem of Sz\H onyi on the number of directions, determined by a point set in $\F_p^2$.

We then focus on $\Aff(\F)$ and present a new incidence bound between a set of points and a set of lines in $\F^2$, which explicitly depends on the energy of the set of lines as affine transformations under composition. This bound, strong when the number of lines is considerably smaller than the number of points, yields generalisations of structural theorems of Elekes and Murphy on rich lines in grids. 

In the special case when the set of lines is also a grid -- relating back to sum-products -- we use growth in  $\Aff(\R)$ to obtain a subthreshold estimate on the energy of the set of lines. This yields a unified way to break the ice in various threshold sum-product type energy inequalities. We show this in applications to energy estimates, corresponding to  sets $A(A+ A)$, $A+AA$ (also embracing asymmetric versions) as well as $A+B$ when $A$ has small multiplicative doubling and $\sqrt{|A|} \le |B|\le|A|^{1+o(1)}$.  \end{abstract}

\section{Introduction and main results} {\em The study of growth and expansion in infinite families of groups has undergone
remarkable developments in the last decade.} This is the opening phrase of the review \cite{HaH}, where we direct  the reader interested in the big picture and references  to connections and applications.

Its fountainhead is the following theorem of Helfgott \cite{HH}.
\begin{theorem}[Helfgott] \label{t:harald} Let $A\subset \SL(\F_p)$ be a generating set. Then either $|A^3:=AAA|\geq |A|^{1+\delta}$ for some absolute real $\delta>0$, or $(A\cup (-A))^k = \SL(\F_p)$, for some absolute integer $k$.
 \end{theorem}
 
Throughout this paper $\F$ is a field, with multiplicative group $\F^*$; our  emphasis is on the prime residue field $\F_p$ of odd and sufficiently large characteristic $p$ (but not ``too large'', so that numerical checks are feasible for smaller $p$), as well as the reals $\R$ (or $\C$, with no essential difference). For a finite set of cardinality $|A|$ in a group $G$,  $A^k$ stands for the $k$-fold product of  with itself (which can be also $kA$ if $G$ is abelian and the usual sumset notation $A+A$ for $k=2$), and we define
 $$
K[A]:= \frac{|A^3|}{|A|}.
 $$
Throughout $K$ will stand, in various contexts, for various growth-related quantities; $A$ is always a finite set of more than one element, the universe changing with the context. Further notation is introduced as it becomes necessary.
 
Kowalski \cite{Ko} proved a quantitative version of Theorem \ref{t:harald}, which claimed that for a symmetric $A$, containing identity $1_G$,  one could take $k=3$ and, most importantly,  
$$K[A] \geq \frac{1}{\sqrt{2}}|A|^{\frac{1}{1512}}\,.$$ An upper bound  $\d \le (\log_2 7 -1)/6 < 0.3012$ on the growth exponent $\delta$ in Helfgott's theorem was shown by Button and Roney-Dougal \cite{Button}, along with some argument that this may be the least upper bound.

It easy to see that symmetry constraints on $A$ are not too restrictive (see, e.g.,  \cite[Lemma 2.2]{HH}),  nor is the assumption that, say  $|A|\geq 100$, that is some ``reasonable'' constant. We do assume this in this paper, and with some dedication one can extract other reasonable constants, buried in the standard symbols $\ll$, $\gg$ and $\sim$.

\medskip
One of the reasons that Hefgott's original work did not offer quantitative bounds on the growth exponent $\delta$ was that part of its argument was based on the at the time state-of-the-art sum-product inequality
\begin{equation}\label{e:sp}
|AA|+|A+A| \geq |A|^{1+\epsilon}\,,
\end{equation}
with some absolute $\epsilon>0$, for all sufficiently large $A\subset \F_p$, with say $|A|<p^{.99}$,  to keep away from the case $AA=A+A=A=\F_p$. 
We use the standard notations, say 
$$
A+A:=\{a+a':\,a,a'\in A\}\,,
$$
in the ratio set case $A/A$ forbidding to division by zero.

Inequality \eqref{e:sp}, having originated in a well-known paper of Bourgain, Katz and Tao \cite{BKT} and Konyagin \cite{Kon} had itself initially lacked a lower bound on $\epsilon$, for it was based at the time on a rather lengthy arithmetic and additive-combinatorial lemmata. Even though this lemmata subsequently got simplified, see e.g. \cite{GK}, the  contrast to the real case was stark: after the 1997 paper  \cite{Elekes2} by Elekes, the geometric incidence approach, largely based on the Szemer\'edi-Trotter theorem became the one of choice and immediately granted in \cite{Elekes2} the value $\epsilon =\frac{1}{4}$ in the sum-product inequality over $\R$.

Sum-product estimates in $\F_p$ have considerably caught up with  $\R$ since 2014, owing to the new  geometric incidence tools stemming from the first author's point-plane theorem \cite{R} -- here the incidence estimate \eqref{f:Misha+_a} -- in particular a theorem by Stevens and de Zeeuw  \cite{SdZ}, here the incidence estimate \eqref{f:14'}. Today one can take any $\epsilon<\frac{2}{9}$ in the sum-product inequality \eqref{e:sp} over $\F_p$ for, say $|A|<\sqrt{p}$, see \cite{RSS}, versus the supremum value of $\epsilon$ slightly exceeding $\frac{1}{3}$ for reals, see in particular \cite{soly}, \cite{KS}. 

Since Helfgott's foundational paper \cite{HH},  the two phenomena -- growth in non-commutative groups and sum-product type growth in fields -- have  been mentioned in folklore as closely related, the latter, in some sense, creating the onset for the former, for one sure does multiply and add scalars when multiplying matrices. The connection is most straightforward the affine group $\Aff(\F)$ and was studied in depth and at length in a recent paper by Murphy \cite{Brendan_rich}. The latter work, in particular, applied incidence geometry tools in $\F_p$ and $\mathbb C$ to derive structural theorems on  {\em lines in grids},  which originate in a series of late 1990s works by Elekes  \cite{Elekes1}, \cite{Elekes3}, \cite{Elekes4}, see Theorem \ref{t:Elekes} below. Furthermore, it was first to reverse the connection, having shown that the latter structural theorems, in turn, yield sum-product inequalities which could embrace three different sets. These inequalities were inferior in strength to those arising from  immediate applications of incidence theorems when the sets were of comparable size but gained comparable strength as the set cardinalities drifted apart. Earlier {\em asymmetric} sum-product inequalities were due to Bourgain \cite{Bourg} and the second author \cite{ISh}; they call for somewhat elaborate dyadic induction schemes with an application of an incidence bound on each step.

\medskip
Kowalski made the proof of Theorem \ref{t:harald} seemingly independent of the sum-product phenomenon.
Over all, his account of growth in $\SL(\F_p)$ reads as more geometric, with emphasis on what is nowadays referred to as dimensional {\em Larsen-Pink type inequalities} \cite{LP} -- here they are Lemmata \ref{l:line}, \ref{l:trace}, which roughly assert that a putative set $A$ with a small $K[A]$ behaves in some ways like a three-dimensional ball, roughly that  one can control its cross-sections by certain straight lines and hyperplanes by the corresponding power of its volume.

\medskip
This study started as an attempt to clarify the relation between geometric incidence bounds of sum-product type  and aspects of growth in the groups $\SL(\F_p)$ and $\Aff(\F_p)$. One of its conclusions is that for the specific purpose of a quantitative version of Theorem \ref{t:harald} in $\SL(\F_p)$ and its simplified version in  $\Aff(\F_p)$, cited as the forthcoming Theorem \ref{t:harald-affine}, one may be better off without, rather than with even full-strength sum-product estimates. Our proofs of stronger quantitative versions of these theorems are based on cruder geometric estimates. Heuristically, one might expect to benefit by addition and multiplication being incorporated into a single group multiplication operation without disassembling the latter.

On the other hand, there is a ``higher level'' of growth questions, namely those, necessitating energy estimates, that is bounding the number of solutions of equations with several variables in $A$, where the connection with geometric incidence estimates is immediate. We address a variety of such questions, in the context of the affine group only, exploring the relationship in both directions. In some instances this involves using as a shortcut the non-commutative Balog-Szemer\'edi-Gowers theorem, which tends to level out quantitative estimates. This is probably too general for our purposes and can potentially be replaced and improved via a direct argument in the affine group but would inevitably make exposition longer and more technical.

\medskip
The first result in this paper is an improvement of Kowalski's bound, which follows the lines of his proof but makes its every component work more efficiently, making more emphasis on the geometric-dimensional aspect of Larsen-Pink bounds. A similar attempt to follow Helfgott's original proof, equipped with today's sum-product type incidence bounds appears to yield a much more demure improvement, if any.
 
\begin{theorem} \label{t:h} Let $A$ be a symmetric generating set of $\SL(\F_p)$, with $|A^3|=K|A|$. Then, if $|A|$ exceeds some absolute constant, either $A^3=\SL(\F_p)$, or $K\gg |A|^{\frac{1}{20}}.$\end{theorem}

Theorem \ref{t:h} enables the following corollary, concerning the diameter\footnote{We thank H. Helfgott for pointing out that owing to the structure of the proof of Theorem \ref{t:h}, the outcome of Corollary \ref{c:h} is almost as strong as if we had $K\gg |A|^{\frac{1}{12}}$  in Theorem \ref{t:h}.} of the Cayley graph of $G=\SL(\F_p)$. We remark that $C$ in the next corollary is ``reasonable'' and can be computed explicitly from constants buried throughout the proof of Theorem \ref{t:h}.

\begin{corollary} \label{c:h} Let $C>1$ be an absolute constant and $A$ symmetric generating set of $G=\SL(\F_p)$ with $N>C$ elements. Let $d=\log_{\frac{3}{2}}8\approx 5.13.$  Then the Cayley graph of $G$, relative to $A$, has diameter at most  ${\displaystyle C \left( \frac{  \log |G| } {\log(N/C)} \right)^{d} } .$ \end{corollary}

\medskip
The only allusion Kowalski makes to the sum-product phenomenon in his proof of Helfgott's theorem is that  at some point (indicated explicitly in the forthcoming proof of Theorem \ref{t:h}) the proof takes advantage of a ``clever observation, the idea of which goes back to work of Glibichuk and Konyagin \cite{GK} on the sum-product phenomenon" (orthography has been changed). The original observation by  Glibichuk and Konyagin has been often referred to as  {\em additive pivot}, and Helfgott's review \cite{HaH} discusses at some length its essence and generalisations.

The aim of the original additive pivot by Glibichuk and Konyagin was, for a set $A\subset \F_p$, with say $|A|\leq \sqrt{p}$, to find some $\xi \in \F_p$, such that $|A+\xi A|\gg |A|^2$ and $\xi$ could be easily expressed algebraically in terms of just several elements of $A$. The ``clever observation'' was that one can take $\xi = 1+$ some element of  $\frac{A-A}{A-A}$.  This trick has been subsequently used and developed in quite a few papers, written on the sum-product phenomenon in $\F_p$ in 2007--2014 prior to \cite{R} and a new generation of estimates it gave rise to, beginning with \cite{RRS}.  Authors of  the former series of papers 
(including both authors of this article)
seem to have been unaware of a beautiful theorem of Sz\H onyi \cite{Sz}[Theorem 5.2] from as early as the mid-1990s, which readily implies that that, in fact, a positive proportion of $\xi \in \frac{A-A}{A-A}$ would have the desired property.

\begin{theorem}\label{t:sz} A non-collinear set $A\subset \F_p^2$ of $1<|A|\leq p$ points determines at least $\frac{|A|+3}{2}$ distinct directions. \end{theorem} 

The estimate of the theorem is sharp, the answer differing from just $|A|$  (for an even $|A|$) in the real case, due to Ungar \cite{Ungar} by roughly a factor of $2$. This  reflects the fact that $\F_p^*$ has sets, not growing under multiplication,  that is cosets of multiplicative subgroups, while over the reals the minimum size of the product set is roughly twice the size of the set. Also note that for $|A|>p$ all the $p+1$ directions are determined simply by the pigeonhole principle, for then the equation $a+l=a'+l'$, with $a,a'\in A$ and $l,l'$ lying on a line through the origin, always has non-trivial solutions.

However, the cost in terms of the outcome sum-product estimate in $\F_p$ of the implementation of the additive pivot trick itself (but for superfluous and somewhat lengthy arguments, to which there is still no replacement in $\F_q$) was eventually brought down to almost nothing, see the range of papers between Garaev's foundational work \cite{Garaev} and \cite{Rudnev}. The latter paper would have had the same quantitative outcome as it does, had its author known Theorem \ref{t:sz}, which would have only made it considerably shorter. 

In contrast, the relative of the additive pivot used in the proof of Theorem \ref{t:h} -- to which we still, following Helfgott, refer to as a pivot -- comes at no technical cost.  If one assumes a suitable  analogue of Theorem \ref{t:sz} to hold in $\SL(\F_p)$, so that one can morph the corresponding parts of the $\SL(\F_p)$ Theorem \ref{t:h} and the $\Aff(\F_p)$ Theorem \ref{t:h-a}, then the exponent $\frac{1}{20}$ in Theorem \ref{t:h} would only improve to $\frac{1}{15}$. On the other hand, as far as the diameter of the Cayley graph is concerned in Corollary \ref{c:h}, the pivot argument in the proof of Theorem \ref{t:h} yields, in fact, a stronger result, namely ``as if'' there were exponent $\frac{1}{12}$ in Theorem \ref{t:h}. See the forthcoming proofs.

\medskip
As an illustration of the additive pivot argument adapted to growth in groups, Helfgott \cite{HaH} discusses the model case of the affine group $\Aff(\F_p)$, its element $(a,b)$  acting on $x\in \F_p$ as $ax+b$, the stabiliser $\Stab(x)$ of $x$  -- geometrically a non-vertical line through the identity $(1,0)$ -- being defined by the condition $ax+b=x$. The following statement is an amalgamation of \cite[Proposition 4.8]{HaH}, and  \cite[Theorem 27]{Brendan_rich} by Murphy.

\begin{theorem}\label{t:harald-affine} Let $A=\{(a,b)\subset \F_p^*\times \F_p\}\subset \Aff(\F_p)$ be symmetric, contain identity, and $a\neq 1$ for some $(a,b)\in A$. Let  $\pi(A)$ be the projection of $A$ on the $a$-axis $\{(\F^*_p,0)\}$ and $K=K[A]$. There exists an integer $C>1$, such that \\
\indent (i)  either  $A\subseteq  \Stab(x)$, for some $x\in \F_p$, \\
\indent (ii) or  $|\pi (A)|\ll K^C$,\\
\indent (iii) or  $|\pi (A)|p \ll K^C|A|$ and $A^C$ contains the $b$-axis $\{(1,\F_p)\}$.\end{theorem}
Theorem \ref{t:harald-affine} is a structural statement, claiming roughly that if $K[A]$ is small, then $A$, as a set of points in the $(a,b)$-plane looks like as follows: either  (i) $A$  lies on a single non-vertical line through the group identity element $(1,0)$, or (ii) $A$ lies on a small number of vertical lines, and (iii) a similar claim if $|A|\gg p$, where the vertical lines tend to get full as $|A|$ gets bigger.

In \cite[Proposition 4.8]{HaH} $C$ is, in fact,  two constants, equal $57$ and $112$. The proposition has an arguably lengthy, although elementary proof, based on the pivot argument, much in the spirit of the original sum-product additive pivot in \cite{GK}. 

In \cite[Theorem 27]{Brendan_rich} $C=20$, and the claim (i) is slightly weaker: a positive proportion of $A$ lies in $\Stab(x)$. The outset of the proof is similar to \cite{HaH}, its quantitative part is based on the new generation of geometric incidence theorems over $\F_p$.

One observation we make in this paper is that  Theorem \ref{t:sz} appears to be ideally suited to yield a much stronger quantitative statement about the affine group than Theorem \ref{t:harald-affine} and used as a ``black box'' reduces the proof of the claims (i) and (ii) of the corresponding forthcoming  Theorem \ref{t:h-a} to but a few lines, while the case (iii), dealing with $|A|\gg p$ follows easily from well-known and optimal up to constants Beck-type theorem for large point sets in $\F_p^2$, which goes back to Alon's 1980s paper \cite{A}. Morally, the improvement we gain is due, once again, to avoiding any explicit ties with the sum-product phenomenon, which both proofs of Helfgott and Murphy relate to, even though in different ways: the former by way of the variation of the additive pivot argument, the latter by using geometric incidence estimates. 

\begin{theorem}\label{t:h-a} Let $A=\{(a,b)\subset \F_p^*\times \F_p\}\subset \Aff(\F_p)$ be symmetric. Let  $\pi(A)$ be the projection of $A$ on the $a$-axis $\{(\F^*_p,0)\}$.  Then \\
\indent (i)  either  $A\subseteq \Stab(x)$, for some $x\in \F_p$, \\
\indent (ii) or, for $0<\epsilon<1$, if $|A|\leq (1+\epsilon) p,$ one has $|\pi (A)|\leq  2K^4$,\\
\indent (iii) or otherwise
$|\pi (A)|\ll_\epsilon  K^3 \frac{|A|}{p}$; in particular, for $|A|>4p$, we have $|\pi (A)|\leq 2K^3 \frac{|A|}{p}$ and $A^8$  contains the $b$-axis $\{(1,\F_p)\}$. \\
\indent Claims (i), (ii) also apply to $\mathbb R$.
\end{theorem}

Above and further throughout the paper we hide powers of an additional parameter, say $\epsilon$ as above,  into the notation $\ll_\epsilon,\gg_\epsilon$. 

\medskip

\subsection*{More on the affine group}


The group of affine transformation $\Aff (\mathbb{R})$ was studied by Elekes, see \cite{Elekes1}--\cite{Elekes4}.
In particular, Elekes proved the following structural result on rich lines in grids.

\begin{theorem}[Elekes]
	Let $\alpha \in (0,1)$, $n$ be a positive integer and there are $n$ lines intersecting $n\times n$ grid in $\mathbb{R}^2$ at most $\alpha n$ points.
	Then either\\
	$\bullet~$ at least $\gg \alpha^C n$ of these lines are parallel, or\\
	$\bullet~$ at least $\gg \alpha^C n$ of these lines are incident to a common point.\\
	Here $C>0$ is an absolute constant.
	\label{t:Elekes}
\end{theorem}

Parallel and concurrent lines are in correspondence with coset families in $\Aff (\mathbb{R})$ and hence Theorem \ref{t:Elekes} is a result on affine transformations. The theorem and underlying ideas were further developed by several authors, having inspired Murphy's work \cite{Brendan_rich}, which deals with $\Aff (\F)$, $\F$ being $\mathbb C$ or finite. Murphy, in particular, used and combined ideas of both Helfgott and Elekes.

We strengthen Theorem \ref{t:Elekes} apropos of the grid $A\times A,\,A\subset \F$.  As usual, additive energy is defined as
$$\E^+(A,B):=|\{(a,a',b,b')\in A\times A \times B\times B:\,a-b=a'-b'\}|\,,$$ 
$\E^+(A,A)$ being shortened as $\E^+(A)$, and multiplicative energy $\E^\times(A)$  is defined similarly.

We show that under the assumptions of Theorem \ref{t:Elekes}, $\E^+(A)$ or $\E^\times(A-s)$, for some shift $s\in \F$, must be close to its maximum value $|A|^3$, and hence the set itself must have additive or multiplicative structure. So not only the set of rich lines is structured as stated by Elekes theorem, but the grid-forming scalar set is also structured. (This answers a question of B. Hanson.) 

Let us formulate a particular case of the result in the $\F_p$-setting, for more details see Corollary \ref{c:Elekes} in the sequel.

\begin{theorem}
	Let $\alpha \in (0,1)$, $A\subseteq \F_p$ be a set and there are $|A| \le p$ lines intersecting $A\times A$ in $\F_p^2$ at most $\alpha |A|$ points.
	Then, for some absolute constant $C>1$, either\\
	$\bullet~$ at least $\gg \alpha^C |A|$ of these lines are parallel and $\E^{+} (A) \gg \alpha^C |A|^3$, or\\
	$\bullet~$ at least $\gg \alpha^C |A|$ of these lines are incident to a common point and for some $s\in \F_p$, one has $\E^{\times} (A-s) \gg \alpha^C |A|^3$.\\
	Moreover, in the former case the set of $y$--intercepts has additive energy $\gg \alpha^C |A|^3$,
	and in the latter case the set of slopes has  multiplicative energy $\gg \alpha^C |A|^3$.
	\label{t:Elekes_new_intr}
\end{theorem}

It is easy to see that both cases are realised: in the first case one can take, say  $A=\{1,2,\dots, n\}$ and the set of lines $y=x+b$, $b \le n/2$; in the second case take $A=s + \{1,2,2^2, \dots, 2^{n-1}\}$ and the lines $y=s + 2^j (x-s)$, with  $j \le n/2$.

We also prove a more technical result, generalising Theorem \ref{t:Elekes} and improving Murphy's \cite[Theorem 24]{Brendan_rich}  in the appendix to this paper, Theorem \ref{t:BSzG_aff}. Although this theorem is not used explicitly for main conclusions of this paper, we feel it might be of independent interest and use elsewhere.


\medskip
Our main question as to the affine group is how growth therein can related to point-line incidence estimates in $\F^2$, in both directions. Given a set $L$ of {\it non-vertical} lines we identify it with a set in $\Aff(\F)$ and define its energy $\E(L)$ as 
$$
\E(L) := |\{l_1^{-1}\circ l_2 = l_1'^{-1}\circ l_2':\,l_1,l_2,l_1',l_2'\in L\}| = \sum_{h} r^2_{L^{-1}L}(h)\,,
$$
with the standard realisation number notation. Similarly, for an integer $k\geq 3$, 
$$
\E_k(L) : = \sum_{h} r^k_{L^{-1}L}(h)\,.
$$
The same higher energy notation applies to the additive and multiplicative energies in the  scalar case $A\subset \F$. Our base statement is the following fairly ``cheap'' incidence theorem: its proof is essentially the Cauchy-Schwarz inequality, followed by standard applications of Szemer\'edi-Trotter type incidence bounds, summarised in the forthcoming Section \ref{s:inc}.

The number of incidences for a point set $P$ and a line set $L$ is defined throughout by
$$
I(P,L) :=|\{(q,l)\in P\times L:\,p\in l\}| \,,
$$
and similarly if $L$ is replaced by other geometric objects.

\begin{theorem}
	Let $A, B\subseteq \F$ be scalar sets and ${L}$  a set of non-vertical lines in $\F^2$. \\
	If $\F = \mathbb{R}$, then
	either ${I} (A\times B, {L}) \ll |B|^{1/2} |{L}|$,  or
	\begin{equation}\label{f:Elekes_new_R}
	{I} (A\times B, {L}) \ll |B|^{1/2} |A|^{2/3} \E^{1/6}({L}) |{L}|^{1/3} \,.
	\end{equation}
	If  $\F = \F_p$, then either ${I} (A\times B, {L}) \ll |B|^{1/2} |{L}| \cdot  \sqrt{\max\{ 1, |A|^2/p \}}$,
	or
	\begin{equation}\label{f:Elekes_new_Fp}
	{I} (A\times B, {L}) \ll |B|^{1/2} |A|^{5/8} \E^{1/8}({L}) |{L}|^{1/2} \,.
	\end{equation}
	\label{t:Elekes_new}
\end{theorem}

To test whether Theorem \ref{t:Elekes_new} can be useful for sum-product type estimates we consider a common special case when $L$ is also a grid, represented by Cartesian product $C\times D$. Clearly, the main task is estimating $\E(L)$. For that one can use a point-plane theorem of the first author \cite{R} for a ``threshold'' estimate \begin{equation}\label{e:thr}\E(L)\ll |C|^{\frac{5}{2}}|D|^3\,,\end{equation} as long as $C$ and $D$ do not differ vastly in cardinality and in the $\F_p$ case are sufficiently small, see the forthcoming Theorem \ref{t:Brendan_new} and Corollary \ref{t:incidences_new}. Substituting this into the bounds of Theorem \ref{t:Elekes_new} and comparing with the standard incidence bounds coming from the Szemer\'edi-Trotter theorem  \cite{ST} over $\R$ and Stevens-de-Zeeuw theorem \cite{SdZ} over $\F_p$, one readily sees that Theorem \ref{t:Elekes_new} is asymmetric: in the case $|A|=|B|=n$ and $|C|=|D|=m$, it becomes stronger than the above results when $m<n^{\frac{2}{3}}$ and remains nontrivial (surpassing from what one gets  merely by Cauchy-Schwarz) as long as $m>n^{\frac{2}{5}}$. (This is in essence why the above Theorem \ref{t:Elekes_new_intr} says something nontrivial when $|L|\sim n,$ although $L$ there is not a grid.)

Thus, in general one can benefit by using Theorem  \ref{t:Elekes_new} for sum-product type questions when the sets involved have inherently different cardinalities. We give some applications in the final Section \ref{sec:further}  and anticipate more to come.

\medskip
Furthermore, in the real (complex) case we succeed in improving the exponent of $|C|$ in the threshold estimate \eqref{e:thr}.  Lemma \ref{l:11/2-c} delivers a {\em subthreshold} energy bound by using (a non-commutative version of) a theorem by the second author, presented here as Theorem \ref{t:E2/E3}, which relies of the non-commutative version of the Balog-Szemer\'edi-Gowers theorem. (In the affine group, with some effort, the parameter dependencies subsumed in its statement can be made explicit.) 



This yields a unified and essentially non-commutable way to establish sum-product type energy estimates, which so far have not been accessible to the existing, in their essence abelian, methods.  We illustrate this by the next two theorems, with a few words of discussion. 

The next theorem is stated in the symmetric case, i.e., involving one set $A$ only, although one can see from its proof that it extends to the non-symmetric case.

\begin{theorem}
	Let $A\subseteq \R$ be a finite set. 
	There exists an absolute constant $c>0$, such that 
\[
	\sum_{x} r^2_{A(A\pm A)} (x) \ll |A|^{9/2 - c} \,,
\]
	and
\[
	\sum_{x} r^2_{AA \pm A} (x) \ll |A|^{9/2 - c} \,.
\]
\label{t:9/2-c}
\end{theorem}
Note that the threshold estimate $O(|A|^{5/2})$ for the above quantities (known for some 20 years) follows readily by the Szemer\'edi-Trotter theorem or even the more general point-plane incidence bound.  However, it has seen no improvement  so far, despite a fair amount of effort, dealing with closely related questions. 

The first estimate of Theorem \ref{t:9/2-c} is the energy version of \cite[Theorem 2.6]{MR-NS}, which established that $|A(A\pm A)| \gg |A|^{3/2 + c}$ (with a ``reasonable'' explicit value of $c$). The second estimate implies the main result of  \cite{RRSS}, namely that  $|AA+A| \gg |A|^{3/2 + c}$ (with a very small  explicit lower bound for $c$). The method of proof of \cite[Theorem 1.4]{RRSS} relies on Solymosi's construction \cite{soly} -- which does not allow replacing the $+$ by the $-$ sign.  Here, however, we present a higher level and more general proof, which relies on the  Szemer\'edi--Trotter theorem,  pure combinatorics and most importantly the affine group growth Theorem \ref{t:h-a} and Lemma \ref{l:11/2-c}, therefore, in some sense, reversing the causal connection between the two phenomena this paper is concerned with. 

\medskip
The next theorem provides a new additive energy bound for $\E^{+} (A,B)$, where  the set $A$ has small multiplicative doubling, that is $|AA|\leq M|A|$, with some $M>1$. 
Heuristically, this scenario is often referred to as FPMS (few products - many sums), see \cite{FPMS} for discussion and  state of the art results, and is believed to be the cornerstone of the Erd\H os-Szemer\'edi conjecture. A threshold estimate $\E^{+} (A,B)\ll_M |A||B|^{3/2}$ has been known for a long time. Here the symbol $\ll_M$ subsumes a power of $M$. In the symmetric case $B=A$ it has been shown by the second author \cite{s_mixed}  that this can be improved by using the commutative Balog-Szemer\'edi-Gowers theorem combined with  so--called "the operator method", see, e.g., \cite{FPMS}, \cite{s_mixed}. The best known quantitative bound is in \cite{FPMS}. However, the asymmetric case has been out of reach by methods developed in the above and other papers -- see  the references contained in \cite{FPMS} -- spanning the effort of roughly past ten years. The following Theorem \ref{t:men} breaks the ice in the pivotal for applications case $\sqrt{|A|} \le |B|\le |A|^{1+o(1)}$, which automatically yields a small improvement to many quantitative bounds that used the threshold asymmetric additive energy bound.



\begin{theorem} \label{t:men}
	Let $A,B \subset \R$ be sets, $|AA| \le M|A|$, and  
	$|A| \le |B|^2$. 
	Then there is an absolute $c>0$ such that 
	\begin{equation}
	\E^{+} (A,B) \ll_M |B|^{5/3} |A|^{5/6-c} \,.
	\end{equation}
\end{theorem}

The latter two theorems, even though they address the commutative sum-product type estimates are due to taking advantage of non-commutativity of the affine group multiplication. For the abelian methods, used so far exclusively to deal with the sum-product phenomenon, they have been out of reach. Hence, non-commutativity appears to offer a new clue to understanding of the sum-product phenomenon.

\begin{remark} We expect that the methods of this paper will lead to further novel sum-product type applications. Consider, for instance the following interesting question, asked by O. Roche-Newton. Suppose $A\subset \R$ has small additive doubling, that is  $|A+A|\leq K|A|$, for some $K>1$. Is it possible to improve a threshold estimate $|AB| \gg_K |A||B|^{1/2},$ for any $B$? Roche-Newton communicated to us that the machinery developed here in Section \ref{sec:further} shall indeed lead to an improvement. \end{remark}

\section{Growth in  $\SL(\F_p)$} \label{s:sl2}

Throughout this section $\F=\F_p$, $G=\SL(\F_p)$. 

\subsection{Background}

To write down elements of $G$ we use, as usual, matrices $g=\bpm a&b \\ c&d\epm$, with $ad-bc=1$ in the standard basis. A change of basis in $\F^4$ arises as a linear map, corresponding to taking a conjugate $h\to ghg^{-1}$, with some $g\in G$.  

Consider lines through the identity (further denoted as $1_G$) in the quadric $G\subset \F^4$.  The tangent space to $G$ at the identity is the affine hyperplane of traceless matrices (invariant to basis changes). The intersection of $G$ with its tangent space at the identity is a cone $\mathcal C$, consisting of conjugates of the unipotent subgroup 
$$
U_0=\left\{\bpm 1&t\\0 & 1\epm,\;\;\; t\in \F\right\}\,.
$$
We will also denote $U_0=l_1$, for it is geometrically a line.  All lines in $\mathcal C$ arise as conjugates of $U_0$ by elements of $G$. The normaliser of $U_0$ is the standard Borel subgroup 
\begin{equation}\label{e:bor}
B_0=\left\{\bpm \gamma &t\\0&\gamma^{-1}\epm,\;\;\; \gamma\in \F^*,\,t\in \F\right\}\,
\end{equation}
 of upper-triangular matrices. Note that the subgroup of lower-triangular matrices is conjugate to $B_0$, via $\bpm0&-1\\1&0\epm$.
 
Even though, clearly, not all $\SL(\F_p)$ matrices have eigenvalues in $\F^*$, for our purposes  it will suffice to conjugate with elements of $G$ only. Further, by a unipotent subgroup we mean the one  in the  form $U= gU_0 g^{-1}$ for some $g\in G$, and by a Borel subgroup the one in the form $B= gB_0 g^{-1},$ for some $g\in G$. Hence, having identified $B$ means knowing $U\subset B$ and conversely.

Elements  $g\in G\setminus \{\pm \mathcal C\}$  are called {\em regular semisimple}. A regular semisimple $g$ lies in a unique maximal torus $T$  - a maximum commutative subgroup containing $g$.
The torus  $T$ is {\em split} if $g$ is diagonalisable in some basis, then $T$ is given by a set of all unideterminant diagonal matrices, in this basis. Otherwise $T$ is {\em non-split} or {\em anisotropic} and given by a set of matrices in the form $\bpm x&-ay\\ y&x\epm$, for some $a\in \F^*$, such that $-a$ is a non-square and $x^2+ay^2=1$.  

It is easy to calculate that the normaliser of a maximal torus in $G$ has twice its size. In the split case it is the union of $T$ with the the set of traceless matrices in the form $\bpm 0&x\\-x^{-1}&0\epm$ in the same basis over $G$, where $T$ is given by diagonal matrices. In the the non-split case  it is the union of $T$ as above with matrices in the form  $\bpm -x& ay\\ y&x\epm, $ with $x^2+ay^2=-1$. It follows, and it is to be used in the proof of Theorem \ref{t:h}, that given a torus $T\subset G$ (with $T\not\subset \mathcal C$),
\begin{equation}
\left| \{hTh^{-1}:\;h\in G\} \right|\gg p^2\label{vol}\,.
\end{equation}
Indeed, two distinct $h,h'\in G$ yield distinct conjugate to $T$ tori, unless $h'h^{-1}$ lies in the  normaliser of $T$.

Given  $g\in G$, its conjugacy class $\Conj(g)$ is the set $\{hgh^{-1}:\,h\in G\}$. If $g$ a regular semisimple element, the conjugacy class is in 1-1 correspondence with the trace value $\tr \,g=\tau$.  See, e.g.,  \cite{Fulton_Harris}. In particular, for $\tau\neq 0$, $g$ is diagonalisable if and only if $\tau = \gamma+\gamma^{-1},$ for some $\gamma\in \F^*$. A traceless $g$ is diagonalisable if $-1=\iota^2$,  $\iota\in\F_p$, that is for odd $p=4n+1$.

If $\tr \,g = \tau=\pm 2$, then say with $\tau=2$, the cone $\mathcal C$ is the union of three conjugacy classes: $\{1_G\}$, the conjugates of $\bpm 1&1\\0&1 \epm$, and the conjugates of $\bpm 1&a\\0&1 \epm$, where $a$ is some non-square.

\medskip
Geometrically, the three-dimensional variety of lines in the quadric $G\subset \F^4$ arise as translates (i.e. cosets) of those in $\mathcal C$ by group multiplication. However, we are interested and will use the notation $l_\gamma$ only for members of a two-dimensional subvariety of lines in $G$, which 
arise as conjugates of cosets of $U_0$ in $B_0$.

That is, for $\gamma\in \F^*$,
\begin{equation}\label{e:line}
l_\gamma := \left\{ \bpm\gamma&t\\0&\gamma^{-1}\epm:\;\;\;t\in \F\right\}\,,
\end{equation}
relative to some basis.

An important  geometric fact that we shall use is that the set of $\SL(\F_p)$-matrices with the same trace value, that is a two-dimensional quadric in the affine hyperplane $a+d=\tau$, with the equation in the standard basis variables $(a,b,c)$
$$
\tau a - a^2-bc =1
$$ 
may contain geometric lines in $\F^4$ only if they are in the form $l_\gamma$.  This is an easy calculation to be done in the forthcoming lemmata.

\medskip


This completes the minimum background for the following self-contained exposition.  The only extra non-trivial fact to be used is the Frobenius theorem on the minimum dimension of a complex irreducible representation of $\SL(\F_p)$ in the proof of Lemma \ref{l:A^3=G}; however, this may be fully skipped if one is interested in vindicating the value of  $\delta=\frac{1}{20}$ in Theorem \ref{t:h} for sufficiently small sets only, say  $|A|\leq p^{\frac{3}{2}}.$ For a particular case of an estimate along the lines of Lemma \ref{l:A^3=G}, obtained by elementary methods, see \cite{ChI}, where no representation theory is used.


\subsection{Lemmata} 
The first lemma, used throughout, is essentially the Ruzsa distance inequality \cite{Ruz}. 

\begin{lemma}\label{l:Ruz} Let $A\subset G$ be symmetric, with $|AAA|=K|A|.$ Then, for $k\geq 4$, $$|A^k|\leq K^{k-2}|A|\,.$$
\end{lemma}
We provide the elementary proof here; proofs of the remaining lemmata are presented later. \begin{proof} Let $k\geq 4$, and $a,a_1,a_2,\ldots a_k\in A$. Since 
$$a_1a_2\ldots a_k = (a_1a_2a)(a^{-1}a_3\ldots a_k), $$ this defines an injection from $A^k\times A$ to $A^3\times A^{k-1}, $ and the claim follows by induction on $k$, with the base case $k=4.$
$\hfill \Box$ \end{proof}

\medskip
The remaining lemmata may be split into two groups. The first group has two statements: Lemma \ref{l:A^3=G}  and Lemma \ref{l:esc}.  Lemma \ref{l:A^3=G}  proves Theorem \ref{t:h} in a ``large set case'', when $|A|$ is a sufficiently large power of $p$, being in some  sense comparable with $|G|\sim p^3$. Then the claim of Theorem \ref{t:h} follows by a different method.  This is, in a sense, the end of the proof of Theorem \ref{t:h}. 

Lemma \ref{l:esc}  provides a starting point to the proof of Theorem \ref{t:h}, stipulating the existence of a regular semisimple element in $A^2$, and therefore a maximal torus $T$, containing the product $h'h^{-1}\neq  1_G, $ for some $h,h'\in A$; this $T$ is said to be {\em involved} with $A$.

\begin{lemma} \label{l:A^3=G} 	Let $p\ge 3$ be a prime, and $A\subseteq G$ satisfies
	\[
	|A| \ge 2 |G|^{\frac{8}{9}} \gg p^{\frac{8}{3}} \,.
	\]
	Then $A^3 = G$.
	Furthermore,
	\[
	|A^{-1} A|, |A A^{-1}|, |AA| \;\;\;\ge\;\; 2^{-1} \min \{ |G|, |A|^2/p^2 \} \,.
	\]
\end{lemma}

\medskip
Using Lemma \ref{l:A^3=G}, for $p^{\frac{5}{2}}\ll |A| \ll p^{\frac{8}{3}}$, one has
$$
Kp^{\frac{8}{3}}\gg K|A|\geq|AA|\gg p^3, 
$$
so $K\gg p^{\frac{1}{3}}\gg |A|^{\frac{1}{8}}$.

Furthermore, using Lemma \ref{l:A^3=G}, say for  $|A| =  p^{2+c}$, with $\frac{1}{7}\leq c\leq \frac{1}{2}$, yields
$$
K \gg p^{c} = |A|^{\frac{c}{2+c}}\geq |A|^{\frac{1}{15}}\,.
$$
In other words, Lemma \ref{l:A^3=G} yields a stronger result  than claimed by Theorem \ref{t:h} for $A$, such that $|A|\geq p^{\frac{15}{7}}$, and in the sequel we assume the contrary: $|A|<p^{\frac{15}{7}}.$

\medskip
The proof of Theorem \ref{t:h} will begin invoking the next statement, whose prototype is \cite[Lemma 3.10]{Ko}, only our claim is stronger, owing to using the ``geometric'' Lemma \ref{l:line} in the proof.
The overall theme is known as {\em escape from subvarieties}.

\begin{lemma} \label{l:esc} Let $A$ be a symmetric generating set of $G$ and $p\geq 5$, and $$|A|>12+16K|A|^{\frac{1}{3}}\,$$ Then $A$ contains a regular semisimple element $g$ with $\tr\,g\neq 0.$\end{lemma}
\begin{remark}\label{constants}It is easy to calculate that the conclusion of Lemma \ref{l:esc} holds, in particular, for $K\leq |A|^{\frac{1}{20}}$ (essentially the converse of the claim of Theorem \ref{t:h})  and $|A|\geq 110.$\end{remark}

\medskip
The geometric part of the proof of Theorem \ref{t:h}  relies on Lemmata \ref{l:line}, \ref{l:trace}, which are special cases of so-called Larsen-Pink inequalities \cite{LP}. Heuristically, they claim that a putative generating set $A$ with a small doubling constant $K$ behaves, in the three-dimensional quadric $G$,  ``like a three-dimensional ball'', that is its intersection with a line (or affine hyperplane, corresponding to a conjugacy class) is roughly the volume to the power $\frac{1}{3}$ (or $\frac{2}{3}$). Larsen-Pink type bounds are very easy in the $G=\SL(\F_p)$  case: it suffices to use that $G$ is a quadric in $\F_p^4$,  without any remotely advanced group theory or algebraic geometry.

The prototype of the next statement is \cite[Lemma 3.15]{Ko}, only we claim a better estimate.
\begin{lemma} Suppose $A$ is a symmetric generating set of $G$, with $|A^3|= K|A|,$ let  $l_\gamma$ be a line in $G$ in the form \eqref{e:line}.
Then
$$
|l_\gamma \cap A^k| \leq 2 |A^{3k+2}|^{\frac{1}{3}}\leq 2 K^k|A|^{\frac{1}{3}}\,. 
$$ 
\label{l:line}
\end{lemma}


We now consider intersections of $A$ with conjugacy classes $C_\tau$, $\tau$ being the trace value.  In the particular case $\tau=\pm 2$, $C_{\pm 2}=\pm \mathcal C$ denotes the union of the three corresponding conjugacy classes, as mentioned in the background section.

We start with an auxiliary geometric statement, followed immediately by a short proof.
\begin{lemma}\label{l:lcc} The only geometric lines in $\F^4$, contained in $C_\tau$ are in the form $l_\gamma$. In particular, if $\tau=\gamma+\gamma^{-1}$ has no solution in $\F$, $C_\tau$ contains no lines.\end{lemma}
\begin{proof}
The statement is clear for $\tau=\pm 2$, with $C_{\pm2}=\pm \mathcal C$. Furthermore, all geometric lines in $G$ are cosets of some line in $\mathcal C$, that is are conjugates of cosets of $U_0$. Calculating
$$
\bpm a&b\\c&d\epm \bpm 1&t\\0&1\epm = \bpm a&at+b \\ c &ct+d\epm
$$
shows that having the trace equal $a+d$ for all $t$ necessitates  $c=0$, which also means $d=a^{-1}$.
$\hfill \Box$ \end{proof}

The prototype of the next lemma is \cite[Lemma 3.14]{Ko}. We restate it and rewrite the proof in an arguably somewhat more geometric way, suitable for our purpose.
\begin{lemma}\label{l:3.14} Suppose, $\tau\neq 0$, $y_1,y_2\neq  1_G$ and $y_1\neq  y_2$. 

Then either (i)
$$|C_\tau\cap y_1 C_\tau\cap  y_2 C_\tau|\leq 2,$$
or (ii) there exists a basis, so that for $\gamma \in  \F^*$, with $\tau=\gamma+\gamma^{-1}$, 
the intersection $C_\tau\cap y_1 C_\tau\cap  y_2 C_\tau$, is either $l_\gamma\cup l_{\gamma^{-1}}$, in which case both $y_1,y_2
\in l_1$, the unipotent subgroup  $l_1$ corresponding to the basis in question, or $C_\tau\cap y_1 C_\tau\cap  y_2 C_\tau = l_\gamma$, in which case one of $y_1,y_2$  lies in the line $l_{\gamma^2}$ (in the same basis) and the other in $l_{\gamma^2}\cup l_1.$
\end{lemma}

\begin{remark} \label{rem:0} If one allows for $\tau=0$, it is easy to develop the forthcoming proof of Lemma \ref{l:3.14} a bit further and see that apart from changing the outset to $y_1,y_2\neq  \pm 1_G$ and $y_1\neq  \pm y_2$, there is  an additional scenario, when the intersection $C_\tau\cap y_1 C_\tau\cap  y_2 C_\tau$ is the second coset of some maximal torus $T$ in its normaliser -- see the background section -- and $y_1,y_2\in T$.\end{remark} 

The latter two statements enable the following claim, whose prototype is \cite[Theorem 3.11]{Ko}. Our estimate is considerably stronger quantitatively. This is the only  Larsen-Pnik type estimate, explicitly used in the proof of Theorem \ref{t:h}. 

\begin{lemma}[Main lemma]
\label{l:trace} Let $A$ satisfy the conditions of Lemma \ref{l:line} and $C_\tau$ be a conjugacy class, $\tau\neq 0$.  Then, for $k\geq 2$, one has the estimate
\begin{equation}\label{e:main} |A^k \cap C_\tau|  \ll  K^{\frac{4k-4}{3}}|A|^{\frac{2}{3}}  + K^k|A|^{\frac{1}{3}}\,, \end{equation}
where the last term is only present when $\tau=\gamma+\gamma^{-1},$ for some $\gamma\in \F^*$. \end{lemma}


\subsection{Proofs}
The group-theoretical aspect of growth in a  non-commutative group $G$, discovered by Helfgott, central to it and seemingly irreplaceable, is the use of the group action on itself by conjugation, which means in the $G=\SL(\F_p)$ case that for some fixed $g\in G\setminus\{1_G\}$, the image of what we call  the {\em Helfgott map} $\_phi:\, a \; \to \;a g a^{-1}$, a projection of $A$ on the conjugacy class of $g$, is two-dimensional. Let $g$ be regular semisimple and $T$ a maximal torus, containing $g$. The key observation is that  the projection $\_phi$ acts one-to-one on $A$, as long as $A^{-1}A$ does not meet $T\setminus \{\pm 1_G\}$. 

Theorem \ref{t:h} readily follows from the {\em pivot} observation, namely that as long as $|A|$ is reasonably small, relative to $|G|$, such an element $g$ can be found already  in the set $AA^{-1}AA^{-1}$. This only uses Lemma \ref{l:esc} and the fact that the symmetric  $A$ generates $G$. Once $g$ has been found, Lemma \ref{l:trace}, namely the map $\psi$ constructed in the proof of the lemma,  guarantees  growth.

The Helfgott map is central for the proofs of Theorem \ref{t:h} and \ref{t:h-a}. In \cite{HH}, \cite{HaH}, \cite{Ko} the map is described via  lemmata alluding to the orbit-stabiliser theorem, based on the following
general and elementary Lemma \ref{l:CS_ineq}. We quote the lemma and its proof for a reader preferring a more structured exposition. However, for a more streamlined presentation, we  have chosen to include the corresponding one-line estimates explicitly in the proofs of Theorems \ref{t:h}, \ref{t:h-a}, without referring to the following Lemma \ref{l:CS_ineq}.


\begin{lemma}
	Let $G$ be any group and $A\subseteq G$ a finite set.
	Then for any $g\in G$,  there is $a_0 \in A$ such that 
	\begin{equation}\label{f:CS_ineq}
	|A| \le |\Conj (g) \cap AgA^{-1}| \cdot |\Stab(g) \cap a_0^{-1} A| \,.
	\end{equation}
	Here $\Conj (g)$ is the conjugacy class and $\Stab(g)$ the stabiliser of $g$ in $G$.
	\label{l:CS_ineq}
\end{lemma}
\begin{proof}
	Let $\_phi : A \to  \Conj (g) \cap AgA^{-1}$ be a map $\_phi(a) := a g a^{-1}$. 
	One sees that $\_phi(a) = \_phi (b)$ iff 
	\[
	b^{-1} a g = g b^{-1} a \,.
	\]
	In other words, $b^{-1} a \in \Stab(g) \cap A^{-1} A$, that is both $a,b$ lie in the same left coset of $\Stab(g)$. 
	Clearly then, since $A\setminus 1_G$ is partitioned by cosets of  $\Stab(g)$,
	and by the pigeonhole (alias Dirichlet) principle, there is $a_0 \in A$ such that 
	\[
	|A| \le |\Conj (g) \cap AgA^{-1}| \cdot |\Stab(g) \cap a_0^{-1} A|\,,
	\]
	%
	as required. 
	$\hfill\Box$
\end{proof}

\bigskip
We will shortly present the proof of Theorem \ref{t:h}, assuming lemmata in the previous section. We follow the structure of the proof in Kowalski's paper \cite{Ko}. However, before we embark on it, let us present the  -- short, assuming the ``Larsen-Pink'' Lemmata \ref{l:line} and \ref{l:3.14} -- proof of the main Lemma \ref{l:trace}. 

\subsection*{Proof of Lemma \ref{l:trace}}
Let $C_\tau$ be a conjugacy class, with the trace value $\tau$ or the union of three classes for $\tau=\pm 2$ (so $C_{\pm 2}$ is, respectively, the cone $\pm\mathcal C$).

Having fixed a basis, so $g=\bpm a&b\\c&d\epm$, $C_\tau$ is the intersection of $G$ with an (affine for $\tau\neq 0$)  hyperplane $a+d=\tau$. For our purpose we can identify $C_\tau$ with a two-quadric, which is $G$ intersecting the hyperplane. 

By Lemma \ref{l:lcc}, $C_\tau$ is a cone for $\tau=\pm 2$, is doubly ruled for $\tau:\,\tau=\gamma+\gamma^{-1}$ by conjugates of some line $l_\gamma$ and contains no lines in $\F^4$ otherwise.

For brevity let us replace $A^k$ with $A$, and set $A_\tau = A\cap C_\tau.$
\medskip
Consider a map
$$
\psi:\; G^3\to G^2,\;\;\;(g_1,g_2,g_3)\to (g_1^{-1}g_2, g_1^{-1}g_3).
$$
Let $\psi_\tau$ the restriction of $\_phi$ to $g_1,g_2,g_3\in A_\tau$, so the image of $\psi_\tau$ lies in $A^2\times A^2$.

Fixing the values $y_1=g_1^{-1}g_2, y_2=g_1^{-1}g_3$ we are interested in the fibre -- the pre-image of $(y_1,y_2)$ under $\psi_\tau$. Since $g_1^{-1}=y_1g_2^{-1}=y_2g_3^{-1}$ and $C_\tau$ (as well as $A$) is symmetric, this means $g_1 \in C_\tau\cap y_1C_\tau\cap y_2C_\tau$.

We now invoke Lemma \ref{l:3.14} to show that unless a positive proportion of $A^k$ lies on some line $l_\gamma \subset C_\tau$, the map $\psi_\tau$ is at worst two-to-one for a positive proportion of $(g_1,g_2,g_3)\in A_\tau\times A_\tau\times A_\tau$. By {\em some line} we mean a conjugate of $l_\gamma$ or $l_{\gamma^{-1}}$ in $B_0$, relative to some fixed basis, for $\tau$ determines the value of $\gamma$, up to inversion ($\gamma=\gamma^{-1}$ for $\tau=\pm 2$).  This is guaranteed unless the image element $(y_1,y_2)$ is such that $|C_\tau\cap y_1C_\tau\cap y_2C_\tau|>2$. This degenerate scenario is
described by claim (ii) of Lemma \ref{l:3.14} and can occur in two ways: let us see what they say about the triple $(g_1,g_2,g_3).$ 

One way for claim (ii) of Lemma \ref{l:3.14} to occur is when both $y_1,y_2$ lie in the same unipotent subgroup $l_1$ of $G$, so that $g_1$, as well as $g_2= g_1 (g_1^{-1}g_2)$ and $g_3= g_1 (g_1^{-1}g_2)$ lie in the same left coset of $l_1$: a line $l_\gamma$ or $l_{\gamma^{-1}}$ in $C_\tau,$ in some fixed basis. This cannot happen for, more than, say, $30\%$ of the triples $(g_1,g_2,g_3) \in A_\tau\times A_\tau\times A_\tau$, unless $50\%$ of $A_\tau$ lies on a single chosen line.  This is what the second term in estimate \eqref{e:main} stands for, using Lemma \ref{l:line}.

The same concerns the other way of claim (ii) of Lemma  \ref{l:3.14}, when in some fixed basis, $g_1$ lies on one of the two lines, which we identify as $l_{\gamma^{-1}}$, and then one of $g_2, g_3$ lies on  the line $l_\gamma$ and the other either $l_{\gamma^{-1}}$ of $l_\gamma$. This is again impossible to hold for a positive proportion of triples $(g_1,g_2,g_3)$, unless a large proportion of $A_\tau$ is collinear. 

If this is not the case, the conclusion is that  the map  $\psi_\tau$ is (at worst) a 2-1 injection of a positive proportion of $A_\tau \times A_\tau \times A_\tau$ into $A^{2k}\times A^{2k}$, and using Lemma \ref{l:Ruz} to bound from above the image size leads straight to estimate \eqref{e:main}.

$\hfill\Box$

\subsection*{Proof of Theorem \ref{t:h}}

Let us call a maximal torus $T\subset G$  (with $T\not\subset  \mathcal C$) {\em involved} with $A$ if  there are $a_1,a_2\in A$, so that $g=a_1^{-1}a_2\in T\setminus\{\pm 1_G\}$ (that is the distinct $a_1,a_2$ both lie in the same left coset of $T$) and $g$ has nonzero trace. Recall that from symmetry $A^2=AA^{-1}=A^{-1}A$.

By Lemma \ref{l:esc}, there exists some maximal torus $T$, involved with $A$:  there is $g = a_1^{-1} a_2$, with $a_1,a_2\in A$, and  $g$ has  trace $\tau\neq 0,\pm 2$. 

We now conjugate $T$ with all elements of $A$, considering the union $\cup_{h\in A} (T'=hTh^{-1})$. If all maximal tori $T'$, arising thereby, are involved with $A$, we continue conjugating each of these tori with elements of $A$. After that, once again, either we get at least one new torus, which is not involved with $A$, or all the tori, generated so far from $T$ are involved with $A$. And so on. Since $A$ generates $G$, the procedure will end in one of the two ways:  either (i) there is some torus $T$ involved with $A$ and  some  $h\in A$, such that $T'=hTh^{-1}$ is not involved with $A$, or (ii) for all $h\in G$ and some (initial maximal torus) $T$, every torus $hTh^{-1}$ is involved with $A$.
Consider the two scenarios separately.

\medskip
{\sf Case (i) -- pivot case.}  
 
The maximal torus $T'$ is not involved with $A$. However, $T=h^{-1}T'h$ is:  there is a non-trivial nonzero trace element $g\in A^{-1}A=A^2,$ lying in $h^{-1}T'h$, therefore $g'=hgh^{-1}\in T'$. and $\tau=\tr\,g\neq 0,\pm 2$. (In other words, $T'$ is not involved with $A$ but is involved with $A^2$.)
 
Consider the projection 
$$
\_phi: \, A\to C_\tau,\;\;\; a\to a g' a^{-1}\in A^6\,,
$$
This projection is one-to-one, for if $a_1,a_2$ have the same image, this means that $a_1^{-1}a_2\in T'$, but $T'$ is not involved with $A$. It follows that $|A^6\cap C_\tau|\geq |A|.$

Applying  Lemma \ref{l:trace} with $k=6$ implies that 

$$
K \gg |A|^{\frac{1}{20}}\,.
$$

 {\sf Case (ii) -- large set case.}  Suppose, for any $h\in G$, all tori $h T h^{-1}$ (not contained in $\mathcal C$) are involved with $A$. The number of  such tori (not meeting, except at $\{\pm 1_G\}$) is, by \eqref{vol}, $\gg p^2$, and (as the worst case scenario) one may assume that $(A^2=AA^{-1})\setminus \{\pm 1_G\}$ is partitioned between these tori.

On the other hand, we can bound from below the individual intersection similar to how it was done in the previous case. To do this, take a torus $T$, a nonzero trace $\tau$ element $g\in A^2\setminus\{\pm 1_G\}$ on it, and consider the projection 
$$
\_phi: \, A\to C_\tau,\;\;\;a\to a g a^{-1}\,.
$$
The cardinality of the image of this map is bounded by $|A^4\cap C_\tau|$, so there is a fibre of cardinality $\gg |A|/ |A^4\cap C_\tau|$, i.e., that for all 
$a_1,a_2$ from $A$ on this fibre, $a_1^{-1} a_2\in T\cap (A^{-1}A=A^2)$. (A fibre of $\_phi$ is the intersection of $A$ with a left coset of $T$).

It follows by \eqref{vol} that 
$$
|A^2|  \ge \sum_{h\in G/N(T)} |A^2 \cap h T
h^{-1}| \gg p^2 \frac{|A|}{|A^4\cap C_\tau|}\,,
$$
where $N(T)$ is the normaliser of $T$ (having twice its cardinality).

Using estimate \eqref{e:main} with $k=4$ to bound the denominator yields
$$
K^5 \gg p^{2} |A|^{-2/3}\,.
$$
As it was discussed following Lemma \ref{l:A^3=G}, the lemma's estimates enable one to assume $|A|\leq p^{\frac{15}{7}}$. 
Combining this with the latter argument yields that if $K\leq |A|^{\frac{1}{18}}$, then $ K \gg  |A|^{\frac{4}{75}}$, this concludes Case (ii), with a better estimate for $K$ than in Case (i),

Choosing the (worse) estimate $K\gg |A|^{\frac{1}{20}}$ of Case (i) concludes the proof of Theorem \ref{t:h}. 

$\hfill\Box$

 \subsection*{Proof of  Corollary  \ref{c:h}}
 We begin with $A$ of cardinality $N$, without loss of generality assuming, say $N<p^\epsilon,$ for some sufficiently small $\epsilon<1$, set $A_0=A$ and start iterating applications of Theorem \ref{t:h} until $A$ reaches, say $\sqrt{|G|}$, so that we do not have to bother with Case (ii) in the proof of Theorem \ref{t:h}, while once $|A^n|\geq \sqrt{|G|}$ for some $n$, we will need only a finite number of additional iterations starting from $A=A^n$ to cover the whole $G$ .  
  
 Observe that we can state the outcome of the proof of Theorem \ref{t:h} on growth on the first  iteration step as $c|A_0|^{3/2}  \leq |A_1:=A^{12}|$, where $c$ is an absolute constant hidden in the Vinogradov symbol. However, on further iterations, numbered by $k+1$, one can use the fact that it is $A_0$ that generates $G$, and therefore the element $g'$, constructed in the proof of Theorem \ref{t:h} lies, instead of $A_k^4$ (where $A_k$ is the -- symmetric -- output set from the $k$th iteration) in $A_0 A_k^2 A_0$. This means, after the second iteration we have, with some absolute $c\in (0,1)$, the estimate $c((cN)^{3/2})^{3/2}\leq |A^{12\cdot 8 + 4}|$, after the third one $c^{1+3/2 + (3/2)^2} N^{(3/2)^3} \leq |A^{(12\cdot 8 + 4)\cdot 8+4}|,$ and so on.
 
 After $k$ iterations, summing geometric progressions, one has, rather crudely,
 $$
 c^{2\cdot (3/2)^k} N^{(3/2)^k} < | A^{2 \cdot 8^k} |\,,
 $$
 whence the claim follows after taking logarithms and estimating $k$ from above when the left-hand side reaches $\sqrt{|G|},$ and subsequently adjusting the constant $C(c)$ if necessary.

$\hfill\Box$

\subsection*{Proof of  Lemma \ref{l:A^3=G}}

	All  statements of the lemma are established similarly, let us begin with the second one. 
	Let $A(x)$ be the characteristic function of $A$ and $f(x) = {A} (x) - |A|/|G|$.
	Clearly, $\sum_x f(x) = 0$.
	Consider the energy
	\begin{equation}\label{e:e}\begin{aligned}
	\E(A) := |\{ a^{-1}b = c^{-1}d ~:~ a,b,c,d \in A \}| & = \frac{|A|^4}{|G|} + \sum_{g\in G} \left( \sum_{x\in G} f(x) f(g x) \right)^2 \\
	& :=  \frac{|A|^4}{|G|} + \E(f).\end{aligned}
	\end{equation}
	The Frobenius theorem \cite{Frobenius} on representations of $G$ gives the following bound (see \cite{Gill}, \cite{Gowers}, \cite{Nikolov})
	for the convolution of any functions $f_1$ and $f_2$ with zero mean:
	\begin{equation}\label{tmp:03.11.2018_1}
	\sum_{g\in G} \left( \sum_{x\in G} f_1 (x) f_2 (g x) \right)^2 \le p^2 \|f_1\|^2_2 \| f_2\|_2^2  \,.
	\end{equation}
	Hence
	\begin{equation}\label{tmp:enr}
	\E(A) \le \frac{|A|^4}{|G|} + p^2 |A|^2 \,.
	\end{equation}
	Using the Cauchy--Schwarz inequality and the last bound, we get
	\[
	|A|^4 \le \E(A) |AA^{-1}| \le \left( \frac{|A|^4}{|G|} + p^2 |A|^2 \right) |A A^{-1}|
	\]
	and similarly for $|A^{-1}A|$ and $|AA|$. 
	Thus, we have proved the second inequality.

	Now, from bound \eqref{tmp:enr}, we get
	\[
		\sum_{g\in G} \left( \sum_{x\in G} A (x) A (g x^{-1}) - \frac{|A|^2}{|G|} \right)^2 = \E (f) \le p^2 |A|^2 \,.
	\]
	Hence if for a certain $\alpha$ one has $AA \cap \alpha A^{-1} = \emptyset$, then
	\[		\begin{aligned}
		|A| \cdot \left( \frac{|A|^2}{|G|} \right)^2 
			& \le
		\sum_{g\in \alpha A^{-1}} \left( \sum_{x\in G} A (x) A (g x^{-1}) - \frac{|A|^2}{|G|} \right)^2 \\
			& \le  
				\sum_{g\in G} \left( \sum_{x\in G} A (x) A (g x^{-1}) - \frac{|A|^2}{|G|} \right)^2
			\;\;\le \;\;p^2 |A|^2 \,,\end{aligned}
	\] 
	so 
	\[
		|A| < p^{8/3} \,.
	\]
	It means that $A^3 = G$, provided that $|A| \ge  p^{8/3}$.  This completes the proof.

	
	$\hfill\Box$

\subsection*{Proof of  Lemma \ref{l:esc}}

	Suppose that the set $\tr\, A$ of traces of elements of $A$ is contained in  $S = \{ \pm 2,0 \}$.
	Then (see \cite[page 70]{Fulton_Harris}) $A$ contains an element, conjugate to one of	$$
	u=\left( {\begin{array}{cc}
		1 & t  \\
		0 & 1  \\
		\end{array} } \right) \,, ~
	v=\left( {\begin{array}{cc}
		-1 & t  \\
		0 & -1  \\
		\end{array} } \right) \,, ~
	w=\left( {\begin{array}{cc}
		0 & 1  \\
		-1 & 0  \\
		\end{array} } \right) \,.
	$$
	The fact ot $t \in \F^*$ being a square or nonsquare identifies different conjugate classes, but this does not matter.
	
	Suppose, in some basis, arising from the standard basis by conjugating with elements of $G$  only, $u\in A$ or $v\in A$ (for some $t\in \F^*$); without loss of generality, it is $u\in A$. Since $A$ is a generating set, there is $g=(ab|cd) \in A$ -- for compactness we use this shorthand notation for $2\times 2$ matrices throughout the proof -- with $c\neq 0$. Recall that $B_0\triangleleft G$ stands for upper-triangular matrices, let
	$A_*= A\cap B_0$.
	Elements of $B_0$ with trace in $S$ lie on the union of four lines $l_{\pm 1},\,l_{\pm \iota}$ (the latter only if $\iota\in \F,\,\iota^2=-1$), and so, by Lemma \ref{l:line}, one has $|A_*| \le 8 K|A|^{\frac{1}{3}}$ and can be regarded as negligible. 
	
	Furthermore, any $g\in A_1:=A\setminus A_*$ has $c\neq 0$. Calculating  $\tr(gu) = a+d+tc  \in S$, $\tr(gu^{-1}) = a+d-tc  \in S$, we conclude that  for  $p\neq 2,3$,  $a+d=0$, that is all elements in $A_1$ have zero trace.

	Conjugating with elements of $G$, we can assume that $w \in A_1$ and as long as $|A_1|>4$ find $g=(ab|c (-a)) \in A_1$ such that $g \neq \pm w$, $g\neq \pm (0\iota |\iota 0),$ (the latter only if $-1$ is a square in $\F$) and hence $a\neq 0$. Indeed,  the above four elements constitute the maximum set of  $\SL(\F_p)$-matrices with zeroes on the main diagonal, containing $g$, such that all pair-wise products of its elements have trace in $S$. These are the only solutions, for  $g=(0b|(-b^{-1}) 0)$ of three quadratic equations 
	$-\tr(gw)=b+b^{-1}\in S$.
	
	So, we take $g=(ab|c (-a)) \in A_1$, with $a\neq 0$.
	One has $\tr(gw) = c-b \in S$ and hence either (i)  $c-b = \pm 2$, or  (ii) $b=c$.
	
	Consider  case (i).  From $\det(g) =1$, we obtain $(b \pm 1)^2 = - a^2$, hence $-1$ is  a square in $\F:$ $\iota\in \F$.	
	It follows that
	$w$ is diagonalisable in $G$ as $x = (\iota 0|0 (-\iota))$.
	
	Conjugating one more time (with an element of $G$) we assume that $x \in A_1$ and again take $g=(ab|c(-a))$ from $A_1$. Once again, if we throw away four elements (in the new basis) in the form $g=(0b|(-b^{-1}) 0)$, we can assume that $a\neq 0$. Indeed, in the same way as we have already had it for $b=1$, for any $b\in \F^*$ the maximum set containing $g=(0b|(-b)^{-1} 0)$ of $\SL(\F_p)$--matrices with zeroes on the main diagonal, such that all pair-wise products have traces in $S$ is $\{(0\beta|(-\beta^{-1}) 0):\,\beta=\pm b,\pm \iota b\}.$
	
	Thus, if $|A_1|>8$ we can assume that there is some $g=(ab|c (-a)) \in A_1$, with $a\neq 0$, 
	and get  $\tr(g x) =2\iota a$. 
	This  is still meant to be in $S$, it follows that $a = \pm \iota$. However, in the latter case ether $b$ or $c=0$, so $g$ lies in the intersection of $A$ with four lines. By Lemma \ref{l:line} the number of such elements is at most $8 K|A|^{\frac{1}{3}}$.
	
	We are done with case (i) now, and pass to case (ii) above, whose input is as follows: there is $A_2\subseteq A$, with
	$$
	|A_2|\geq |A|-16 K|A|^{\frac{1}{3}}-8,
	$$
	such that each element of $A_2$ has the form $(ab|b(-a))$, with $a\neq 0$. Geometrically, this can be viewed as follows:  points $(a,b)\in \F^2$ lie on the circle $a^2+b^2=-1$. For two elements $ (ab|b(-a))$ and $(\alpha \beta |\beta(-\alpha))$, the condition of having the trace of their product lie in $S$ translates to $(a,b)\cdot(\alpha,\beta)\in \{0,\pm 1\}$, as the dot product of two-vectors in $\F^2$. It is easy to verify (essentially in the same way one does it for the unit circle in $\mathbb R^2$) that given a point $(a,b)\in \F^2$ on the circle, that is with with $a^2+b^2=-1$, the maximum set of points $(\alpha,\beta)$ on the circle, including $(a,b)$ itself, and such that all pair-wise dot products equal $0$ or $\pm 1$ is $\{\pm (a,b), \pm(-b,a)\}.$
		
	Thus if $|A|>12+16K|A|^{1/3}$ there is a regular semisimple, with non-zero trace element in $A^2$, as claimed.
	$\hfill\Box$

\subsection*{Proof of Lemma \ref{l:line}}

Fix the basis, consider, for some $\gamma\in \F^*$ a line $l_\gamma=\left\{\bpm \gamma&t\\0&\gamma^{-1}\epm,\;\;\;t\in \F\right\}.$ Since $A$ is a generator of $G$, it contains a fixed  element 
$$
h = \left(\begin{array}{ll}a&b\\c& d\end{array}\right)\,\mbox{ with }c\neq 0\,,
$$
for otherwise $A$ would generate only  upper-triangular matrices in this basis.

Now consider a map
$$
\psi:\;\;\left\{ \begin{array}{llll} l_\gamma\times l^*_\gamma \times  l_\gamma &\to&  G\,,\\
(g_1,g_2,g_3)&\to& x_1 (h g_2 h^{-1}) g_3. \end{array}\right.\,,
$$
with  $ l^*_\gamma $ being the line  $ l^*_\gamma $ without one point $g_2=g_2^*$ on it to be specified. E.g., if $\gamma=1$, then, one cannot possibly have $g_2=1_{G}$, for $g_1g_3$ will lie on the line $ l_{\gamma^2}$, and $\psi(g_1,1_{ G},g_3):\, l_\gamma\times l_\gamma\to  l_{\gamma^2}$ cannot be injective. 

Assuming that $\psi$ is injective yields the claim of the lemma, since if $g_1,g_2,g_3\in A$, one has  $g_1 (h g_2 g^{-1}) g_3 \in A^{3k+2}$.

\medskip
The rest of the proof is a calculation:  setting $\delta=(1-\gamma^{-2})$ one has
\begin{equation}\begin{aligned}
h g_2 h^{-1} & =  \left(\begin{array}{ll}a&b\\c& d\end{array}\right) \left( \begin{array}{ll} \gamma & t_2\\ 0&\gamma^{-1}\end{array}\right) \left(\begin{array}{ll}d&-b\\-c& a\end{array}\right) \\
& = \left(\begin{array}{ll}\gamma(1+\delta bc)-ac t_2& -\gamma\delta ab + a^2t_2
\\ \gamma\delta cd - c^2t_2 & \gamma^{-1} - \gamma\delta bc + ac t_2\end{array}\right).\end{aligned}
\label{int}\end{equation}
Observe that the element under the main diagonal is zero when
$$
t_2  = \gamma\delta c^{-1}d :=t_2^*,
$$
determining the above-mentioned matrix $g_2^*$ to be thrown out.

If $t_2=t_2^*$, then $\gamma(1+\delta bc)-ac t_2^*=\gamma^{-1},$ so $h x_2^* h^{-1} \in l_{\gamma^{-1}}$.
Therefore $\psi(g_1,g_2^*,g_3):\,l_\gamma\times l_\gamma\to l_{\gamma}$ cannot be injective. Indeed, since $h g_2^* h^{-1} \in l_{\gamma^{-1}}$, and both $g_1,g_3\in l_\gamma$, then $g_1 (h g_2^* h^{-1}) g_3 \in  l_{\gamma}$.

\medskip
Other than that, assuming that $t_2\neq t_2^*$, the map $\_phi$ is easily seen to be injective. It suffices to calculate 
$$\begin{aligned}
g_1 (h g_2 h^{-1})  & = \left( \begin{array}{ll} \gamma & t_1\\ 0&\gamma^{-1}\end{array}\right) \left(\begin{array}{ll}\gamma(1+\delta bc)-ac t_2& -\gamma\delta ab + a^2t_2
\\ \gamma\delta cd - c^2t_2 & \gamma^{-1} - \gamma\delta bc + ac t_2\end{array}\right) \\
& =  \left(\begin{array}{ll}\gamma^2 (1+\delta bc)- \gamma ac t_2+t_1(\gamma\delta cd - c^2t_2) & \star
\\ \gamma^{-1}(\gamma\delta cd - c^2t_2) & \star \end{array}\right) \,.
\end{aligned}$$
Hence, knowing 
$$
g_1 (h g_2 h^{-1}) = \left(\begin{array}{ll}u&v\\w&s \end{array}\right) = y
$$
defines $t_2$ in terms of $w\neq 0$, then $t_1$ gets defined in terms of $u$, and finally  $g_3=(g_1 (h g_2 h^{-1}))^{-1}y.$

$\hfill\Box$

\subsection*{Proof of Lemma \ref{l:3.14}}

Consider a matrix $g=\bpm x&y\\u&v\epm\in G,$ with trace $x+v=\tau\neq 0$. Let $y_1=\bpm a&b\\c&d\epm\in G\setminus\{1_G\}$. Suppose, $g\in y_1C_\tau$, this means $\tr(y_1^{-1}g)=\tau,$  that is, eliminating $v$.
\begin{equation}\label{str}
(d-a)x-cy-bu+\tau(a-1)=0\,.
\end{equation}

Furthermore, for $y_2\neq 1_G,y_1$ (for $\tau =0$ one clearly needs the conditions $y_{1,2}\neq\pm 1_G, \,y_1\neq\pm y_2$)
the intersection $C_\tau\cap y_1 C_\tau \cap y_2 C_\tau$ is the intersection of $G\subset \F^4$ with three -- affine for $\tau\neq 0$ -- hyperplanes. The hyperplanes defining $C_\tau$ and $y_1 C_\tau$ coincide only if equation \eqref{str} is vacuous, which only happens if $b=c=0$, $a=d$ and either $a=1$ (so $y_1=1_G$) or $\tau=0$, which allows for $y_1=-1_G$. Hence, the hyperplanes  defining $y_1 C_\tau$ and $y_2 C_\tau$ may coincide only if $y_1=y_2^{-1}$ (and $y_1=-y_2^{-1}$ for $\tau=0$.

Therefore, the three hyperplanes (only two of which may coincide)  intersect either along a line or a two-plane. If the intersection is a line in $\F^4$, it  either meets $G$ in at most two points or lies in $G$. The former case constitutes claim (i) of the lemma, and at this point we are done with it. In the latter case, by Lemma \ref{l:lcc}, the line is some $l_\gamma$, this may happen only for $\tau:\,\tau=\gamma+\gamma^{-1}$.

The rest of the proof belongs to claim (ii).

First, let us deal with the degeneracy when the three hyperplanes in question meet along a two-plane: this is where $\tau=0$ would be special, so suppose $\tau\neq 0$.
 Let equation \eqref{str} be defined by some $(a,b,c,d)$ as a plane in $(x,y,u)$-variables ($v$ having been eliminated by the constraint that $\tr\,g=\tr (y_1^{-1}g)=\tau$). Let us describe, given $y_1$, the set of other matrices $y_2=\bpm a'&b'\\c'&d'\epm\in G$ , so that also $\tr (y_2^{-1}g)=\tau$, that is replacing $(a,b,c,d)$ with $(a',b',c',d')$ in equation \eqref{str} determines the same plane in the $(x,y,u)$-variables.

This means, projectively one must have $$(a-d:c:b:\tau(1-a))= (a'-d':c':b':\tau(1-a'))\,,$$ 
hence, since \footnote{It is an easy exercise to describe what exactly may happen for $\tau=0$ as well, see Remark \ref{rem:0}.} $\tau\neq 0$, for some $\lambda\neq 0$, one has
$$
\bpm a'&b'\\c'&d'\epm = \bpm 1-\lambda(1-a)&\lambda b \\ \lambda c& 1-\lambda(1-d)\epm\,.
$$
Equating the determinant to $1$ and using $ad-bc=1$ yields $\lambda=1$, that is $y_2=y_1$ or $a+d=2$. In the latter case 
$y_1\in \mathcal C$, 
so in some basis we have $a=d=1$, and both $y_1,y_2\in l_1$, a unipotent subgroup.

Hence, for $y_1\neq y_2$ and $\tau\neq 0$, one has $C_\tau\cap y_1 C_\tau = C_\tau\cap y_1 C_\tau \cap y_2 C_\tau$ only if $y_1,y_2$ both lie in the same unipotent subgroup. In the basis, where $a=d=1$, $c=0$ this means by \eqref{str} that $u=0$, so $xv=1$, plus $x + v=\tau$ that is
\begin{equation}\label{two}C_\tau\cap y_1 C_\tau = C_\tau\cap y_1 C_\tau \cap y_2 C_\tau =  l_\gamma\cup l_\gamma^{-1},\end{equation} with $\gamma+\gamma^{-1}=\tau$.
This scenario abides with claim (ii) of the lemma.

It remains to consider the case when $C_\tau\cap y_1 C_\tau \cap y_2 C_\tau$ is a single line (that is the three hyperplanes in question intersect along a line which happens to lie in $G$). In some basis the line is given as $l_\gamma$, with $\gamma+\gamma^{-1}=\tau.$  Set $x=\gamma,y=t,u=0$, getting
\begin{equation}
(a-d)\gamma +\tau(1-a)+ct=0\,.
\label{strr}\end{equation}
For this to be valid for every $t$, one must have $c=0$, so $g$ lies in the same Borel subgroup as $l_\gamma$. If $\tau \neq 0$, then  
either $a=d=1$ or $a=\gamma^2$. In other words, given a basis, the line $l_\gamma$ in this basis (regardless of the basis $\gamma+\gamma^{-1}=\tau$ may lie in $C_\tau\cap y_1 C_\tau $ only if either $y_1$ lies in the unipotent subgroup $l_1$, contained in the Borel subgroup hosting $l_\gamma$ and  corresponding to shift along the line $l_\gamma$ or if $y_1 \in l_{\gamma^2}$, in which case clearly $l_\gamma = y_1 l_{\gamma^{-1}}$.

The same clearly concerns $y_2$. If both $y_1,y_2$ lie in the unipotent subgroup, then, as has been shown, we have \eqref{two}. Other than that, the intersection is a single line $l_\gamma$, with $y_1,y_2\in l_{\gamma^2}\cup l_1,$ but not both in $l_1$.

 This concludes the proof of the lemma.

$\hfill\Box$


\section{Affine group}
\label{sec:aff}

Throughout  this section $G$ is the group of invertible affine transformations $\Aff (\F)$ of a field $\F$, i.e., maps of the form  $x \to ax + b$, $a\in \F^*, b\in \F$. Thus $G$ can be identified with the set of $2\times 2$ matrices $\bpm a&b\\0&1\epm$,   where $a\in \F^*, b\in \F$, with matrix multiplication, or just pairs $(a,b)\in \F^*\ltimes \F$, with semidirect product multiplication  $(a,b) \cdot (c,d) = (ac, ad+b)$ and  identity $1_G=(1,0)$. It is isomorphic to the standard Borel subgroup $B_0$, see \eqref{e:bor}, considered in Section \ref{s:sl2}.

The group $\Aff (\F)$ contains the standard unipotent subgroup $U_0=\{(1,b):\,b\in\F\}$ -- which is normal -- and the standard dilation subgroup $T_0=\{(a,0):\,a\in\F^*\}$, so $G=U_0\rtimes T_0$. A maximal torus $T$ is a subgroup $\Stab(x)$, for some $x\in \F$, a conjugate of $T_0$, defined by the condition $ax+b=x,$ hence $\Stab(x) =\{(a,x(1-a)):\, a\in \F^*\}$. Its elements commute, hence it is also the centraliser of each of its elements. For the centraliser of $g=(a,b),\,a\neq 1$ we may use the notation $C(a,b) = \Stab(\frac{b}{1-a})$,  if $a=1$, then the centraliser of $(1,b)$, for $b\neq 0,$ is $U_0$.

Observe that geometrically, viewing elements $(a,b)\in G$ as lines $y=ax+b$ in $\F^2$, $U_0$ is the set of parallel lines with unit slope, its cosets are sets of parallel lines of given slope, $\Stab(x)$ is, naturally, a set of lines concurrent at $(x,x)$.

\subsection{Incidence theorems}\label{s:inc}
In this section we quote the necessary incidence results. We remind the reader the  Szemer\'edi-Trotter \cite{ST} estimate that if $P$ is a finite set of points in the real or complex plane, than the number $L_k$ of lines, supporting, for $1<k\leq |P|$, at least $k$  points  of $P$ is bounded as
\begin{equation}\label{e:ST}
L_k \ll \frac{|P|^3}{k^3} + \frac{|P|}{k},\qquad\mbox{equivalently}\qquad I(P,L)\ll (|P||L|)^{\frac{2}{3}}+|P|+|L|\,.
\end{equation}
We use this only when $P$ is a Cartesian product. This special case was  addressed by
Solymosi and Tardos  \cite{SoT} who showed, in particular, that  constants, hidden over $\mathbb C$, are ``reasonable'' and comparable to those over $\mathbb R$. 

\medskip
The remaining incidence bounds are in the positive characteristic case. The next two apply to  sufficiently large sets of finite fields, for our purposes just $\F_p$. One is Alon's \cite[formula (4.2)]{A} version of Beck's theorem, claiming the following.
If $P$ is a set of points in the projective plane over $\F_p$, with $|P| > (1 + \epsilon)(p + 1),$  for some $0<\epsilon<1$, then the set $L(P)$ of lines, determined by pairs of points of $P$ has cardinality
\begin{equation}\label{e:Alon} 
|L(P)| \geq \epsilon^2\frac{1-\epsilon}{2+2\epsilon}(p+1)^2\,.
\end{equation}

A closely related result about incidences in $\F_p^2$ is due (among others) to Vinh \cite{Vinh}: if $P,L$ are, respectively, sets of points and lines, then the number of incidences satisfies the asymptotic estimate

\begin{equation} 
\left| I(P,L) -\frac{|P||L|}{p}\right| \leq \sqrt{p|P||L|}\,.\label{e:Vinh}\end{equation}

\medskip
The remaining two incidence estimates cover sufficiently small sets in positive characteristic (and any sets in zero characteristic), their $\F_p$-versions are as follows. If $P,\Pi$ are, respectively, sets of points and planes in $\F_p^3$, with, say $|\Pi|\geq |P|$ and maximum number of collinear points $k$, then 

\begin{equation}\label{f:Misha+_a}
	{I} (P, \Pi)  - \frac{|P| |\Pi|}{p} \ll |P|^{1/2} |\Pi| + k |\Pi| \,.	
	\end{equation}

The above estimate implies the best result on point-line incidences in $\F_p$ which is due to Stevens and de Zeeuw~\cite{SdZ}. We need it in its strongest case of the point set being a Cartesian product. Namely if
 $A,B \subset \F_p$ are two scalar sets and ${L}$ a collection of lines in $\F^2_p$, the number of incidences
	${I}(A\times B,{L})$ is bounded as follows:
	\begin{equation}\label{f:14'}
	{I}(A\times B, {L}) - \frac{|A| |B| |{L}|}{p} \ll |A|^{3/4}|B|^{1/2}|{L}|^{3/4} +|A||B|+|{L}|.
	\end{equation}

\subsection{Proof of Theorem \ref{t:h-a} and further results}
In this section we prove Theorem \ref{t:h-a}, strengthening Theorem \ref{t:harald-affine} in the introduction. The proof does not use what we refer to as {\em sum-product type} incidence estimates \eqref{f:Misha+_a} and \eqref{f:14'}, but rather Theorem \ref{t:sz} and estimates \eqref{e:Vinh} and \eqref{e:Alon} for the easy case of $|A|\gg p$.

Furthermore, Theorem \ref{t:h-a} admits the forthcoming  Corollary \ref{c:el}, which gets formulated in energy terms. To pass to the corollary  we use, as a black box,  the non-commutative Balog-Szemer\'edi-Gowers theorem. Corollary \ref{c:el} implies a variant of the Elekes Theorem \ref{t:Elekes}, see the forthcoming Remark \ref{r:el}.

We proceed by another application of Theorem \ref{t:h-a}, which, instead of the horizontal projection $\pi(A)$ (relative to the notation $g=(a,b)\in A\subset G$)  deals with the vertical projection. 
This is stated by Theorem \ref{t:Brendan_new}, whose proof uses Theorem \ref{t:h-a} and  the  incidence estimate \eqref{f:Misha+_a}.

\subsubsection{Proof of Theorem \ref{t:h-a}}
Let $A\subseteq G$ be symmetric, identified with $\{(a,b)\} \subset \F_p^*\times \F_p.$ Let $L(A)$ be the set of lines in $\F_p^2$ defined by pairs of distinct points of $A$.


First, consider the case $|A|\leq p$. Suppose the set  $AA^{-1}=AA$ lies on a single line. 
Then by symmetry of $A$, that such a line can only pass through the identity $(1,0)$. Furthermore, $A$ itself then also lies on a line, and by symmetry this is a line through the identity. Indeed, if a line, containing the whole of $A$ is a coset $cH$, where $H$ is $U_0$ or a torus $T$, then $A$ also lies in $Hc^{-1}$, which is the same line only if $c\in H$. A line through $(1,0)$ corresponds to a maximum proper subgroup of $G$, and then there is nothing left to prove. 

Otherwise, let $D$ be the set of slopes of lines in $L(A)$. By Theorem \ref{t:sz}, $|D|\geq\frac{|A|+3}{2}.$
Furthermore, for $g=(a,b)\in A^{-1}A=AA$, with $a\neq 1$ (which then exists) consider a map
\begin{equation}
\_phi_g:\;\;\;h\in A\,\to hgh^{-1}.
\label{e:map_g}\end{equation}
Geometrically, this map is projection of $A$ on a vertical line $a=const$ through $g$, along lines with the slope $-x=\frac{b}{a-1}$. Indeed, if two distinct $h,h'$ have the same image by $\_phi_g$, this means, $h^{-1}h'\in \Stab(x):=T_g$. Equivalently both $h,h'$ lie in the same coset $cT_g$, which is geometrically some line in $L(A)$, with the slope $-x$. For any slope $-x\in D$, there is a suitable $g\in AA$, that is such that there is a line through $g$ in $L(A)$ with slope $-x$.


Consider the collection of maps $\_phi_g$, with one $g$ representing each non-vertical slope in $D$, $T_g$ being the centraliser of $g$. The set   $AA\setminus \{1_G\}$ is  partitioned between $|D|$ maximal subgroups, namely the collection of sets  $T_g$ plus  the unipotent subgroup. Hence, there is a  maximal torus $T_{g_*}$, for some $g_*\in AA\setminus U_0$, supporting some, but at most $\frac{|A^2|}{|D|-1}$ non-identity elements of $AA$.  The torus $T_{g_*}$ is {\em involved} with $A$ in the language of the proof of Theorem \ref{t:h}, but {\em not} ``very involved'', since no line with the slope $-x_*$, corresponding to $T_{g_*}$ will support more than $\frac{2K|A|}{|A|+1}\leq 2K$ points of $A$, although there is one such line with at least two points of $A$.

Therefore, the maximum fibre size of $\_phi_{g_*}$ is $\frac{2K|A|}{|A|+1}$.

It follows that the image $\_phi_{g_*}(A)$ has cardinality 
$$|\_phi_{g_*}(A)|\geq \frac{|A|+1}{2K}.$$
Furthermore, since $hg_*h^{-1}g_*^{-1}$ is unipotent, it is easy to see that
$$|A\_phi_{g_*}(A)|\geq |\pi(A)||\_phi_{g_*}(A)| \geq |\pi(A)|\frac{|A|+1}{2K}\,.$$
This inequailty simply reflects the fact that multiplying a set of $(a,b)$ with, say $n$ different values of $a$ with $(1,b')$ with $m$ distinct values of $b'$ one gets at least $mn$ distinct pairs.

Combining the  latter inequality with Lemma \ref{l:Ruz}, namely
\begin{equation}\label{e:A^4}
K^3 |A|\ge|A^5g_*^{-1}| \ge  |\pi(A)|\frac{|A|+1}{2K}\,,
\end{equation}
proves claim (ii) of Theorem \ref{t:h-a} if $|A|\leq p$, but the claim extends to, say $|A|\leq 2p$, since for a set $A\subset \F^2_p$ with $|A|>p$ points, one has $|D|=p+1$. 

The argument so far also applies to reals if one replaces the estimate for $|D|$ by the well-known one, due to Ungar \cite{Ungar}.

\medskip
It remains deal with large sets over $\F_p$. Let us first address the case $|A|>4p$. Consider the complement $L^c(A)$ of $L(A)$. Comparing the incidence bound \eqref{e:Vinh} with $I(A,L^c(A))\leq |L^c(A)|$
 yields
$$
|L^c(A)|\leq \frac{4}{9}p^2,
$$
hence 
$$
|L(A)|> \frac{5}{9}p^2.
$$
It follows that there is a non-vertical direction with more than $\frac{p}{2}$ parallel lines in $L(A)$.
Hence for some 
$g_* \in A$ and map $\_phi_{g_*}$  as in \eqref{e:map_g}, one concludes, similar to \eqref{e:A^4}, that  
$$K^3|A|\geq |A\_phi_{g_*}(A)| \geq |\pi(A)||\_phi_{g_*}(A)|>|\pi(A)|\frac{p}{2}\,.$$ 
Furthermore, for any $h\in A$, the commutator  $hg_*h^{-1}g_*^{-1} \in U_0$, and therefore by the Cauchy-Davenport theorem, the product $(AgAg^{-1} )(gAg^{-1}A)\subseteq A^8$ contains the unipotent subgroup $U_0=\{(1, \F_p)\}$.

Finally, for sets $A$, whose cardinality $(1+\epsilon)p <|A|\leq 4p$ one can use Alon's estimate \eqref{e:Alon}, which yields $|L(A)|\gg_\epsilon p^2$ to settle the rest of claim (iii) of Theorem \ref{t:h-a}.





\subsubsection{Further results}
Let  $\rho(A)$ be the ``vertical'' projection of $A\subseteq G$, namely
$$
\rho(A) =\{b:\,\exists g=(a,b)\, \in \,A\}\,.%
$$
Also  set
\[
w = w(A) := |A|^{-3} \max_{\alpha \in G}\, \E(A\cap \alpha U_0)
\]
and
\[
w_* = w_*(A) := |A|^{-3} \max_{\alpha \in G}\, \max_{g\notin U_0}\, \E(A\cap \alpha \mathrm{C}(g)) \,.
\]

As we will see below the meaning of $w,w_*$ is simple: if they are large, then $A$ contains a large subset of pairs $(a, b) \in A$ such that either 
the set of $b$ has large additive energy or 
the set of $a$ has large multiplicative energy.

\begin{corollary}\label{c:el} There exists an absolute constant $c\in (0,1)$, such that  for any $A\subset G$, with $|A|\leq 2p$,  one has
	\begin{equation}\label{f:Brendan_new_energy}
	\E(A) \ll |A|^3 \cdot
	\max \{ w, w_*\}^{c}   \,,
	\end{equation}
	the energy $\E(A)$ having been defined by \eqref{e:e} for any group $G$.
 \end{corollary}
 
 \medskip\begin{proof}
	Suppose, for some $M\geq 1$, $\E(A) = |A|^3/M$.
	By the non--commutative  Balog--Szemer\'edi--Gowers Theorem, see \cite[Theorem 32]{Brendan_rich} or \cite[Proposition 2.43, Corollary 2.46]{TV}
	there is $a\in A$ and $A_* \subseteq a^{-1}A$, $|A_*|\gg_M |A|$ such that $|A^3_*| \ll_M |A_*|$.
	Here the signs $\ll_M$, $\gg_M$ mean that  all dependences on $M$ are polynomial. If $A_*$ is not symmetric,  replace it by $A_*\cup A_*^{-1}$, which may only cause losing several extra powers of $M$ (this follows by using Lemma \ref{l:Ruz}, see e.g., \cite[Lemma 2.2]{HH}).

	If $A_* \subseteq \mathrm{C} (h)$ for an element  $h\notin U_0$, then 
	$$\frac{|A|^3}{M} = \E(A)  \ll_M \E(A_*) \le w_* |A|^3\,,$$
	thus $M^C \gg w^{-1}_*$ for some absolute constant $C\geq 1$. 
	
	Otherwise,  applying claim (ii) of Theorem \ref{t:h-a} yields $|\pi(A_*)| \ll_M  1.$ It follows by definition of the quantity $w$,  that again
	$$\frac{|A|^3}{M}=\E (A) \ll_M \E(A_*) \ll_M w|A|^3\,,$$
	and this completes the proof.
	
$\hfill\Box$
\end{proof}

\begin{remark}\label{r:el}
	In \cite[Theorem 1]{Elekes1} Elekes proved that if $A, B \subseteq \Aff (\mathbb{R})$, $|A|, |B| \ge n$ and $|AB| \le Kn$, for some $K\geq 1$, then
	there are $A'\subseteq A$, $B'\subseteq B$, $|A'| \gg K^{-C} |A|$, $|B'| \gg K^{-C} |B|$, $C\geq 1$ is an absolute constant  such that either \\
	$\bullet~$ both $A'$, $B'$ consist of parallel lines, or \\
	$\bullet~$ both $A'$, $B'$ consist of concurrent lines.\\
	It is easy to see that $A$, $B$ from the Elekes' result have comparable sizes and hence our arguments allow to estimate the common energy $\E(A,B)$ as in \eqref{f:Brendan_new_energy}.
	Thus, we have reproved  Theorem 1 from \cite{Elekes1} in $\mathbb{R}$ as well as in $\F_p$, provided that $Kn\ll p$.
\end{remark}

The next theorem is also  closely related to Theorem \ref{t:h-a}, although its proof uses both incidence estimate \eqref{f:Misha+_a}  and Theorem \ref{t:h-a}. As above, for transparency of statements we content ourselves with the case $|A|\ll p$ only. Also, from now on  the symbols $\lesssim,\,\gtrsim$ extend, respectively, $\ll,\,\gg$ to hiding powers of logarithms of set cardinalities involved.
\begin{theorem}\label{t:Brendan_new}
	Let $A\subseteq G$ be symmetric, with  $K=K[A]$ and $|A|\leq 2p$. 
	Then either  \\
	$\bullet~$ ${\displaystyle K \gtrsim |A|^{\frac{1}{5}}\,,}$ or  \\
	$\bullet~$ ${\displaystyle |\rho(A)|\gtrsim \frac{|A|}{K^{2}}}\,$.
\end{theorem}
Note that Theorem \ref{t:h-a} certainly implies that for $|A|\leq 2p$, $|\rho(A)|\geq \frac{|A|}{2K^4}.$

\bigskip


\begin{proof}
%
%
	Set (in line with the notation in the proof of the forthcoming Lemma \ref{l:energy_L_prod})  $$C=\pi(A)\,=\{a:\,\exists g=(a,b) \,\in \,A\}\,,\qquad D=\rho(A)\,.$$  	Consider the energy of $A$, which by symmetry of $A$  is
	$$\E(A):= |\{ xy = zw ~:~ x,y,z,w\in A \}|\,.$$

	By the dyadic pigeonhole principle and the general properties of energy, based on the Cauchy-Schwarz inequality (see, e.g. \cite[Inequality (4.18), Exercise 4.2.1]{TV}) there exists a popular set $D'\subseteq D$ and a number $\Delta\geq 1$, with the property that  
	$$\forall b\in D',\, \D\leq |\{a:\,(a,b)\in A\}|\leq 2\D\,,$$
	and 
	$$
	\E(A)\leq L^4 \E(A'),\qquad\mbox{with }\;L=\lceil \log_2|A|\rceil,\;\;A'=\{(a,b)\in A:\,b\in D'\}\,.
	$$
	We further suppress $L$ by writing $\E(A)\gtrsim \E(A')$. Clearly, $|D'|\D\leq 2|A|$.

	The quantity $\E(A')$ equals the number of solutions of the system of equations \[
	\Big\{
	\begin{array}{cc}
	a_1 a_2 = a'_1 a'_2  \\
	a_1 b_2 + b_1 = a'_1 b'_2 + b'_1  \\
	\end{array}  \,, \quad \quad (a_1, b_1), (a_2, b_2), (a'_1, b'_1), (a'_2, b'_2)  \in A' \,.
	\]
	
	A solution of the second equation can be interpreted as a point-plane incidence, apropos of the set of  at most  $2|A| |D'|$ planes  $x b_2 + y = a'_1 z + b'_1$ and the set of at most $2|A| |D'|$ points, defined by triples $(a_1,b_1,b'_2)$.
	For  a given solution of the second equation, there are at most $2\D$ solutions of the first one.
	
	Applying the incidence estimate \eqref{f:Misha+_a} 
	(the maximal number of collinear points does not exceed $|C|+|D'|$)
	and the Cauchy--Schwartz inequality to estimate the energy $\E(A)$ from below yields
	\begin{equation}  \label{e:long}
	\begin{aligned}
	\frac{|A|^3}{K}	\le \frac{|A|^4}{|AA|} \le \E(A)  & \lesssim   |\D| \left( \frac{|A|^2 |D'|^2}{p} + |A|^{3/2} |D'|^{3/2} + |A| |D'|^2 + |A| |C| |D'| \right)\\
	& \ll
	 |A|^{5/2} |D|^{1/2} + |A|^2 |C| \,.\end{aligned}
	\end{equation}
	
	If the second term dominates in the latter estimate, then
	$$
	K|C|\gtrsim |A|,
	$$
	and using claim (ii) of Theorem \ref{t:h-a} yields $K\gtrsim |A|^{\frac{1}{5}}.$
	
	Note that claim (i) of Theorem \ref{t:h-a} implies that $|D|=|A|.$ Furthermore, if the first term dominates in estimate \eqref{e:long}, one has
	$$
	K^2|D|\gtrsim |A|\,,
	$$ 
	and this completes the proof of  Theorem \ref{t:Brendan_new}.
	$\hfill\Box$
\end{proof}

\subsection{Proof of Theorem \ref{t:Elekes_new}}
We  now turn to the bound on the number of incidences of lines and points in $\F^2$, which takes into account the energy of the set of lines as members of $G=\Aff(\F)$, stated in Theorem \ref{t:Elekes_new} in Introduction. Its proof invokes point-line incidence bounds in Section \ref{s:inc}, and therefore we distinguish between $\F=\F_p$ and $\R$ (or equivalently for our purposes $\mathbb C$; we will not mention $\mathbb C$ explicitly further in the sequel).

We will then consider some implications of Theorem \ref{t:Elekes_new}. It will allow for a short proof of
Theorem \ref{t:Elekes}, stated in  Introduction, which  will follow by combining Corollary \ref{c:el} above with the forthcoming Corollary \ref{c:Elekes}. 

Since we are entering the realm of counting the number if solutions of  linear equations, with variables in scalar sets,  $A$ is no longer in $\Aff(\F)$: instead $A,B,C,D$, etc., are finite sets in $\F$, while sets of affine transformations are identified with sets of non-vertical lines in $\F^2$, denoted as $L$. 

As a notation of choice, we use the representation function notation $r_{AB} (x)$ for the 
 number of ways $x$ can be expressed as a  product $ab$ with $a\in A$, $b\in B$, where $A,B$ are sets in some group,
in particular $r_{A+B} (x)$ for addition in $\F$.

We now prove Theorem \ref{t:Elekes_new}.
\begin{proof}[Proof of Theorem \ref{t:Elekes_new}]
	Set $\sigma = {I} (A\times B, {L})$.
	By the Cauchy--Schwarz inequality 	\begin{equation}\label{f:sigma_inc}
	\sigma^2 \le |B| \sum_h r_{{L}^{-1} {L}} (h) \sum_{x\in A} A(hx) \,.
	\end{equation}
	By the pigeonhole principle (since $\sum_h r_{{L}^{-1} {L}} (h)=|L|^2$)  one can assume that the summation in $h$ in \eqref{f:sigma_inc} above is taken over a popular set $\Omega \subseteq \Aff (\F_p)$, where $\sum_{x\in A} A(hx) \ge \D$, with 
	$$\D := \frac{ \sigma^2}{2|B| |{L}|^2}\,.$$
	Further assume that 
	$\D \gg 1$, 
	for  otherwise we are done with the trivial estimate $\sigma \ll |B|^{1/2} |{L}|$.
	
	Then, by the Cauchy--Schwarz inequality
	\begin{equation}\label{f:sigma^4}
		\sigma^4 \ll |B|^2 \E({L}) \sum_{h\in \Omega} \left( \sum_{x\in A} A(hx) \right)^2 \,.
	\end{equation}
	Using 
	the Szemer\'edi--Trotter incidence estimate  \eqref{e:ST} over $\mathbb{R}$, we obtain \eqref{f:Elekes_new_R} in the usual way. Namely denoting, for $k\in \mathbb N$, $\Delta_k=2^{k} \D$ with $\D_0 =\D$, let the sets $\Omega_k$ of ``rich lines'' in $A\times A$  be defined  similarly to how $\Omega$ has been defined. Namely, a line identified with $h\in\Omega_k$ supports  {\em approximately} (that is up to the factor of $2$, rather than $\geq$) $2^k\frac{\sigma^2}{2|B| |{L}|^2}$ points of $A\times A$. From  \eqref{e:ST} we have
	\begin{equation}\label{e:der}
	|\Omega_k|\D_k \sim  \sum_{h\in \Omega_k}  \sum_{x\in A} A(hx) \ll |\Omega_k|^{\frac{2}{3}} |A|^\frac{4}{3} + |A|^2\,,
	\end{equation}
	whence, since naturally $\D_k\leq |A|$,
	\begin{equation}\label{e:derr}\sum_k \sum_{h\in \Omega_k} \left( \sum_{x\in A} A(hx) \right)^2 \sim \sum_k \Delta_k \sum_{h\in \Omega_k}  \sum_{x\in A} A(hx) \ll \sum_k \frac{|A|^4}{\D_k} \ll \frac{|A|^4}{\D} \ll \frac{|A|^4|B||L|^2}{\sigma^2}\,, \end{equation}
	which together with \eqref{f:sigma^4} yields \eqref{f:Elekes_new_R}.
	
	Similarly over $\F_p$ we apply incidence estimate \eqref{f:14'}.  Then the analogue of \eqref{e:der} becomes
	$$
	|\Omega_k|\D_k \sim  \sum_{h\in \Omega_k}  \sum_{x\in A} A(hx) \ll  \frac{|A|^2 |\Omega_k|}{p} +  |\Omega_k|^{\frac{3}{4}} |A|^\frac{5}{4} + |A|^2\,,
	$$ whence, if the $p$-term in the right-hand side can be disregarded, one easily obtains an analogue of \eqref{e:derr} as follows:
	$$
	\sum_k \sum_{h\in \Omega_k} \left( \sum_{x\in A} A(hx) \right)^2   \ll \frac{|A|^5}{\D^2} \ll \frac{|A|^5|B|^2|L|^4}{\sigma^4}\,, 
	$$ and hence \eqref{f:Elekes_new_Fp}.

	Otherwise, if for some $k\geq 0$ the $p$-term dominates the right-hand side, this means $\D_k\ll\frac{|A|^2}{p}$, thus by definition of $\D$
	\[
	\sigma^2 \ll \frac{|B| |{L}|^2 |A|^2}{p} \,.
	\]
	This completes the proof.
	$\hfill\Box$
\end{proof}

\bigskip



We can apply Theorem \ref{t:Elekes_new} to get a lower bound on the size of the image set of a set of affine transformations, alias
the neighbourhood of the set  $A$
in the correspondent Schreier graph, namely, $$\Im_{L} (A) := \{ h(a) ~:~ a\in A,\, h \in {L} \}\,.$$
The next Corollary \ref{c:L(A)} says, in particular, that if $|{L}| \sim |A|$ and ${L}$ has few parallel and concurrent lines,
then $|\Im_{L} (A)|$ can be estimated nontrivially from below.

\begin{corollary}
	Let $A\subseteq \F$ be a set and ${L} \subseteq \Aff (\F)$.
	If $\F = \mathbb{R}$, then
	\[
	|\Im_{L} (A)| \gg \min\{ |A|^2,   (|A|^2 |{L}|^4 \E^{-1} ({L}) )^{1/3} \} \,.
	\]
	If $\F = \F_p$, then
	\[
	|\Im_{L} (A)| \gg \min\{ p, |A|^2,  (|A|^3 |{L}|^4 \E^{-1} ({L}) )^{1/4} \} \,.
	\]
	\label{c:L(A)}
\end{corollary}
\begin{proof}
	Let $B = \Im_{L} (A)$.
	Then  ${I} (A\times B, {L}) = |A| |{L}|$, and for an upper bound one can apply Theorem \ref{t:Elekes_new}.
	If ${I} (A\times B, {L}) \ll |B|^{1/2} |{L}|$, then $|\Im_{L} (A)| \gg |A|^2$ and we are done.
	Similarly, if ${I} (A\times B, {L}) \ll |B|^{1/2} |{L}| \sqrt{|A|^2/p}$, then we obtain
	$|\Im_{L} (A)| \gg p$. 
	Otherwise, applying inequalities \eqref{f:Elekes_new_R}, \eqref{f:Elekes_new_Fp} of Theorem \ref{t:Elekes_new}
	 completes the proof.
	$\hfill\Box$
\end{proof}

\bigskip


As the last result in this section, we study the case when the number of incidences between the point set $A\times A$ and a set of lines $L$ is close to maximum possible value, which according to Theorem \ref{t:Elekes_new} is $|A|^{\frac{1}{2}}|L|$ (for $|A|\leq \sqrt{p}$ in the $\F_p$-case.) In this case we show that there is arithmetic structure not only apropos of the set of lines $L$, but $A$ as well.
The next corollary immediately implies Theorem \ref{t:Elekes_new_intr},  stated in Introduction.

\begin{corollary}
	Let $A \subseteq \F$, with $\F = \mathbb{R}$ or $\F = \F_p$ be a set and ${L}$ be a set of non-vertical lines in $\F^2$, with $|{L}| \ge |A|$.
	Suppose that $K\geq 1$ and  the number of incidences
	
	\begin{equation}\label{f:cor_Elekes}	
	{I} (A\times A, {L}) \ge \frac{|A| |{L}|}{K} \gg  |A|^{1/2} |{L}| \cdot \sqrt{\max\{ 1, |A|^2/p \}}\,,	\end{equation}
	 the term $|A|^2/p$ applying only to the $\F_p$-case.
	
Then there exists an absolute $C\geq 1$, such that  either\\
	$\bullet~$ $\E^{+} (A) \gg |A|^3K^{-C}$ and $\E^+(B),\E^+(A,B)\gg |A| |{L}|^2 K^{-C}$, where $B$ is the set of $y$--intercepts of the lines in  $L$, or\\
	$\bullet~$ there is $s\in \F$ such that  $\E^{\times} (A-s) \gg |A|^3 K^{-C}$ and $\E^\times(D),\E^\times (A-s,D)\gg |A| |{L}|^2 K^{-C}$, where $D$ is the set of slopes of the lines in  $L$.\\
	\label{c:Elekes}
\end{corollary}
\begin{proof}
	Let
	\[
	\frac{|A| |{L}|}{K} \le I(A\times A, {L})  = \sum_{l\in {L}} \sum_{x \in A} A(lx) \le 2 \sum_{l\in {L}_*} \sum_{x \in A} A(lx) \,,
	\]
	where ${L}_* = \{ l\in {L} ~:~ |l\cap (A\times A)| \ge |A|/(2K) \}$ is a popular set of lines.
	Clearly,  $|{L}_*| \ge |{L}|/2K$.
	By Theorem \ref{t:Elekes_new}, applied under the assumptions in the statement of the corollary, it follows that 
	\[
	\E ({L}_*) \gg \frac{|{L}|^4}{|A| K^6} \ge \frac{|{L}|^3}{K^6}
	\]
	over $\mathbb{R}$ and, similarly, $\E ({L}_*) \gg |{L}|^3K^{-8}$ over $\F_p$.
	
	In both cases,  applying Corollary \ref{c:el} yields that for some $g =(\alpha,\beta)\in \Aff (\F)$, one has 
	$|gU_0 \cap {L}_*| \ge |{L}_*|/K^C$
	or $|g \mathrm{C}(h) \cap {L}_*| \ge |{L}_*|/K^C$, where $h\notin U_0$,
	with some absolute constant $C\geq 1$.
	
	Let us consider the former case $|gU_0 \cap {L}_*| \ge |{L}_*|/K^C$.
	Parametrise the intersection $S := gU_0 \cap {L}_*$ by pairs $(\alpha, \beta)$ and denote by $B \subseteq \F_p$ the set of all such $\beta$.
	Here $\alpha\neq 0$ is fixed and $\beta$ runs over a set $B$ of cardinality $|S|$.
	By definition of the set ${L}_*$,  we get
	\begin{equation}\label{f:mixed_AS}
	\frac{|A| |S|}{2K} \le \sum_{l\in S} \sum_{x \in A} A(lx) = \sum_{\beta \in B} \sum_{x \in A} A(\alpha x +\beta) = \sum_{\beta \in B} r_{A-\alpha A} (\beta) \,.
	\end{equation}
	The Cauchy--Schwartz inequality, used twice, yields
	\[
	\frac{|A|^3 }{K^{C+3}} \le \frac{|A|^2 |{L}|}{K^{C+3}} \ll \frac{|A|^2 |S|}{K^2} \ll \sum_\beta r^2_{A-\alpha A} (\beta) \le \E^{+} (A) \,.
	\]
	THis gives the required bound after changing  $C\to C+3$.
	Similarly, \eqref{f:mixed_AS}, again via the Cauchy--Schwartz inequality, implies 
	\[
	\frac{|A| |S|^2}{K^2} \ll \E^{+}(A,B)
	\le \E^{+} (A)^{1/2} \E^{+} (B)^{1/2}\,,
	\]
	and hence the first bullet claim of the corollary, increasing $C$ if necessary.

	Now let $|g \mathrm{C}(h) \cap {L}_*| \ge |{L}_*|/K^C$, where $g$ and $h\notin U_0$ are fixed elements of $G$.
	Then we can parametrise the intersection $S := g\mathrm{C}(h) \cap {L}_*$ by elements $(a m,\, a c (1-m) + b)$
	with fixed $a,b,c$ and $m \in D \subseteq \F^*$, with $|D|=|S|$.
	As above
	\[\begin{aligned}
	\frac{|A| |S|}{2K} \le \sum_{l\in S} \sum_{x \in A} A(lx) & = \sum_{m \in D} \sum_{x \in A} A(a m x + a c (1-m) + b) \\
	& =
	\sum_{m \in a D} r_{\frac{A-(ac + b)}{A-c}} (m)\,.
	\end{aligned}
	\]
	It follows that  for translates of $A$ by $s=c$, or $s=ac+b$, one has
	\[
	\frac{|A|^3 }{K^{C+3}} \le \E^\times (A-s)\,, 
	\]
	as well as
	\[
	\frac{|A| |T|^2}{K^2} \ll \E^{\times}(A-s, D) \le \E^{\times}(A-s)^{1/2} \E^\times (D)^{1/2} \,.
	\]
	This completes the proof.
	$\hfill\Box$
\end{proof}

\subsection{Application of Theorem \ref{t:Elekes_new} to sum-product type incidence questions}
\label{sec:further}

This final section  develops some applications of Theorem \ref{t:Elekes_new}. We focus on the case when the set of lines $L$ is itself a grid $C\times D\subset \F^*\times \F$, so that its energy can be estimated rather efficiently, based on the procedure employed in the proof of Theorem \ref{t:Brendan_new}. This leads to several restatements of the incidence bound, in terms of various energies of $C$ and $D$. To estimate these energies we use the incidence results quoted in Section \ref{s:inc} and their fairly well-known implications. This leads to, for $\F=\R$ and when $|C|\sim|D|$, to an improvement of the general bound for the energy of lines appearing in he proof of Theorem \ref{t:Brendan_new}, stated as Lemma \ref{l:11/2-c}.
 
The proof of Lemma  \ref{l:11/2-c} is based on a structural theorem of the second author, Theorem \ref{t:E2/E3}, which is a generalisation of the Balog-Szemer\'edi-Gowers theorem. Namely, not only does a set in a group contain a large subset with controlled growth when its energy is closed to maximum, but under a more relaxed condition that two of its energies find themselves in a certain critical relation. This turns out to be exactly the case in the putative scenario when  the energy estimate for $\E(L)$ from Theorem \ref{t:Brendan_new} is worst possible, namely when $C$ is nearly closed under multiplication.  The fact that we are able to give an independent estimate on the third moment $\E_3(L)$, so that the two energy estimates are in a critical relation leads to a contradiction by invoking Theorem \ref{t:h-a}. Thus growth in $\Aff(\R)$ leads to new, stronger sum-product type estimates in Theorem \ref{t:9/2-c} stated in Introduction, whose proof concludes this section.


Recall the notation $\E^\times_k (A) = \sum_{x} r^{k}_{A/A} (x)$ and similarly, say as in \eqref{e:e} above for $k=2$, define $\E^\times_k (f)$ for any function $f$ by weighing each solution of the defining equation  $a^{-1}_1b_1 = \ldots = a^{-1}_kb_k$
with variables in $A$, by the value of the product  $f(a_1) \ldots f(b_k).$

As the first preliminary result, the next lemma gives upper bounds for energy of a grid of affine transformations. The lemma and its implication Corollary  \ref{t:incidences_new} it entails also contain estimates in the special case when the set $C$ of slopes has small additive doubling. We do not use these bounds for our main results, however present them, expecting that they find applications in other sum-product type questions.

\begin{lemma}
	Let $(C,D) \subseteq \F^*\times \F$ and  ${L}$ be the set of affine transformations in the form $(c,d)$ or $(c, cd)$, in $\Aff(\F)$, with $c\in C,\, d\in D$.
	Then
	\begin{equation}\label{f:energy_L_prod}
	\E ({L}) \le  \min\{\E^\times_3 (C)^{2/3} \E^{\times}_3 (r_{D-D})^{1/3},  \E^\times_4 (C)^{1/2} \E^{\times}_2 (r_{D-D})^{1/2} \} \,.
	\end{equation}
	If $\F=\F_p$, then for  $|C| \le |D|^2$, one has
	\begin{equation}\label{f:energy_L_prod_rough}
	\E ({L}) \ll \frac{|C|^3 |D|^4}{p} + |C|^{5/2} |D|^3 
	\,,
	\end{equation}	
	The same bound without the $p$-term holds in zero characteristic.
	
	If in addition,  $|C+C| = K|C|$, $|D| \le p^\frac{2}{3}$, and $K|C|^{12} \le p^8$, for some $K\geq 1$, then 
	\begin{equation}\label{f:sm_A+A_F_p}
		\E ({L}) \lesssim K^{5/2} |C|^2 |D|^3 + |C|^3 |D|^2 \,,
	\end{equation}
	and for $\F = \mathbb{R}$,
	\begin{equation}\label{f:sm_A+A_R}
		\E ({L}) \lesssim K |C|^2 |D|^3 + |C|^3 |D|^2\,.
	\end{equation}
	\label{l:energy_L_prod}
\end{lemma}
Note that all bounds of Lemma \ref{l:energy_L_prod} do apply to $\F=\R$ simply by disregarding the $p$-terms and constraints. We further use just $C\times D$ for the grid of lines, concerning both input forms of the lemma, for the proofs are identical.

\medskip
\begin{proof}
	Let us consider the case, when the transformations are in the form $(c,d)\in C\times D$, the other case is similar.
	By the H\"older inequality
	\begin{equation}\label{f:energy_L_prod'}\begin{aligned}
	\E ({L})  & = | \{ c/a = c'/a',\, (d-b)/a = (d'-b')/a' ~:~ a,a',c,c' \in C,\, b,b', d,d'\in D \}| 
	\\
	&= \sum_s r^2_{C/C} (s) r_{(D-D)/(D-D)} (s) \\ & \le \min \{ \,\E^\times_3 (C)^\frac{2}{3} \E^{\times}_3 (r_{D-D})^{1/3},\,\E^\times_4 (C)^{1/2} \E^{\times}_2
	 (r_{D-D})^{1/2}\,\}\,, \end{aligned}
	\end{equation}
	which proves estimate \eqref{f:energy_L_prod}.
	Observe that there is also a negligible number  $|C|^3 |D|^2$ of trivial solutions with  $d=b$, $d'=b'$.

	Bounding trivially $r_{C/C} (s)\leq|C|$ in \eqref{f:energy_L_prod'} we obtain
	\[
	\E ({L}) \le |C| | \{ c (d_1-d_2) = c' (d'_1-d'_2) ~:~ c,c'\in C,\, d_1,d_2,d'_1,d'_2 \in D \} |  \,.
	\]
	The latter  equation can be interpreted as point-plane incidences, where the set of $|C||D|^2$ planes is defined by the formula $c(x-d_2) = yd'_1 - z$,
	and the correspondent set of points has the same cardinality. 
	Hence using  incidence estimate \eqref{f:Misha+_a}, we get 
	\[
	\E ({L}) 
	\ll \frac{|C|^3 |D|^4}{p} + |C|^{5/2} |D|^3 \,.
	\]

	Bound \eqref{f:sm_A+A_R} follows from following the well-known estimates over $\mathbb R$:
	
	$$\E^{\times} (C) \ll  K^2 |C|^2 \log|C|\,,\;\;\;\E^{\times}_2 (r_{D-D}) \ll |D|^6 \log|D|\,,
	$$ see, respectively, \cite{soly} and, e.g, \cite{MR-NS}.

	It remains to prove bound  \eqref{f:sm_A+A_F_p}.
	It suffices to estimate the quantity $\E^{\times}_3 (r_{D-D})$ in \eqref{f:energy_L_prod'}, for by \cite[Lemma 21]{MPORS}, one has (as a corollary of the Sevens-de Zeeuw incidence bound \eqref{f:14'})
\begin{equation}\label{tmp:23.11_1} 
	\E^\times_3 (C) \lesssim K^{15/4} |D|^3 \,, 
\end{equation}
	provided that $K|C|^{12} \le p^8$. 
	
	By the Cauchy---Schwarz inequality, it follows that
\[	
	\E^{\times}_3 (r_{D-D}) = \sum_x \left( \sum_{d,b\in D} r_{(D-d)/(D-b)} (x) \right)^3 \le |D|^4 \sum_{d,b\in D} \sum_x r^3_{(D-d)/(D-b)} (x) = |D|^4 \Q(D) \,,
\]
	where $\Q(D)$ is the number of collinear quadruples in the set $D\times D \subset \F^2$. 
	Combining an asymptotic formula for $\Q(D)$ from \cite[Theorem 10]{MPORS} (a corollary of the Sevens-de Zeeuw incidence bound \eqref{f:14'}) namely,
\[
	\Q(D) - \frac{|D|^8}{p^2} \ll |D|^5 \log |D|  
\]	
	with bound \eqref{tmp:23.11_1}, we obtain the required estimate \eqref{f:sm_A+A_F_p}. 
	This completes the proof.
	$\hfill\Box$
\end{proof}

\bigskip

Combining  Theorem \ref{t:Elekes_new} with Lemma \ref{l:energy_L_prod} gives a new bound for the number of incidences between a set of points and a set of lines in $\F^2$, when both sets are  grids. The next statement arises from the estimates of Theorem  \ref{t:Elekes_new} directly by substituting therein the $\E(L)$ bounds from Lemma \ref{l:energy_L_prod}.

In fact, the following Corollary \ref{t:incidences_new} represents what we would regard as {\em threshold bounds}, namely soon thereafter we shall focus on improving (in the real case, when we can) one of its main bounds \eqref {f:incidences_new_R}.


\begin{corollary}
	Let $A,B,C,D \subseteq \F$ be sets,  $0\not\in C$.
	
	If  $|C| \le |D|^2$, then for $\F=\F_p$, with $|C| |D|^2 \le p^2$, one has
	\begin{equation}\label{f:incidences_new}
	{I} (A\times B, C \times D) \ll |A|^{5/8} |B|^{1/2} |C|^{13/16} |D|^{7/8} + |B|^{1/2} |C| |D| \cdot \sqrt{\max\{ 1, |A|^2/p \}} \,.
	\end{equation}
	If  $|C| \le |D|^2$, then for $\F=\R$, one has
	\begin{equation}\label{f:incidences_new_R}	
	{I} (A\times B, C \times D) \ll |A|^{2/3} |B|^{1/2} |C|^{3/4} |D|^{5/6} + |B|^{1/2} |C| |D|  \,.
	\end{equation}	
	Suppose $|C+C|\le K|C|$, for some $K\geq 1$. Then for $\F=\F_p$ and  $|D| \le p^\frac{2}{3}$, $K|C|^{12} \le p^8$, 
	\begin{equation}\label{f:incidences_new'}
	{I} (A\times B, C \times D) \lesssim K^{5/16} |A|^{5/8} |B|^{1/2} |C|^{3/4} |D|^{7/8} + |B|^{1/2} |C| |D| \cdot \sqrt{\max\{ 1, |A|^2/p \}} \,,
	\end{equation}
	and for $\F=\R$, 
	\begin{equation}\label{f:incidences_new_R'}	
	{I} (A\times B, C \times D) \lesssim K^{1/6} |A|^{2/3} |B|^{1/2} |C|^{2/3} |D|^{5/6} + |B|^{1/2} |C| |D| \,.
	\end{equation}	
	\label{t:incidences_new}
\end{corollary}

Observe that bound  \eqref{f:incidences_new} is better than  the incidence estimate \eqref{f:14'} (for sufficiently small sets, relative to $p$),
provided that $|A|^2 \gg |C| |D|^2$ and $|A|^3 \gg |C| |D|$
(we compare the main term in \eqref{f:14'}
 with the one in \eqref{f:incidences_new}).
The most effective choice of $A,B,C,D$ in Corollary \ref{t:incidences_new} is obviously  $|B| \gg |A| \gg |C| \gg |D|$.

Corollary \ref{t:incidences_new} yields new threshold bounds on convolutions of sets with small multiplicative doubling.
Moreover, an application of Theorem \ref{t:9/2-c} enables, for $|Q|\sim |A|$ a slightly stronger estimate over $\R$.   

\begin{corollary}
	Let $A\subset \F_p^*$, with $1 < |A| \le p^\frac{2}{3}$,  and $Q$ be another set, such that $|A|^{5} |QA|^{6}\le p^{8}$.  
	Then for any $z\neq 0$ one has 
	\begin{equation}\label{f:subgroup_Stepanov+}
	r_{Q-Q} (z) \ll |QA|^\frac{9}{8} |A|^{-\frac{5}{16}} \,,
	\end{equation}
	and for any $R \subseteq \F$, $|R| = |A|$ the following holds 
	\begin{equation}\label{f:subgroup_Stepanov+'} 
	\E^{+} (Q,R) \ll |QA|^\frac{5}{4} |A|^\frac{11}{8} \log |A| \,.
	\end{equation}
	If $\F = \mathbb{R}$, then
	\begin{equation}\label{f:subgroup_Stepanov+'_R} 
	\E^{+} (Q,R) \ll |QA|^\frac{4}{3} |A|^\frac{7}{6} \log |A| \,.
	\end{equation}
	\label{c:subgroup_Stepanov+}
\end{corollary}

For example, it is known that if $Q\subseteq \mathbb{R}$ and $|QA| \leq M|Q|$, for some $M\geq 1$, then $\E^{+} (Q,A) \ll_M |Q|^{3/2} |A|$, yet our new inequality \eqref{f:subgroup_Stepanov+'_R} is always better. A similar situation  takes place in $\F_p$ where  \eqref{f:subgroup_Stepanov+'} is better than $\E^{+} (Q,A) \ll |Q|^{3/2} |A|$ in the case when $|Q| \gtrsim |A|^{3/2}$. 

\bigskip
\begin{proof}
	To obtain  \eqref{f:subgroup_Stepanov+} observe that 
	\[
	r_{Q-Q} (z) \le |A|^{-2} |\{ q_1 a^{-1}_1 - q_2 a^{-1}_2 = z  ~:~ a_1, a_2 \in A,\, q_1, q_2 \in QA \}| 
	= |A|^{-2} {I} (QA\times QA, A \times A) \,,
	\]
	where the set of lines $x a^{-1}_1 - y a^{-1}_2 = z$ of size $|A|^2$  is a Cartesian product and  $QA\times QA$ is the set of points. 
	Using  Corollary \ref{t:incidences_new} and the assumption $|A|^{5} |QA|^{6}\le p^{8}$ yields
	\[
	r_{Q-Q} (z) \ll |QA|^{9/8} |A|^{-5/16} + |QA|^{1/2} \cdot \sqrt{\max\{ 1, |QA|^2/p \}}
	\ll |QA|^{9/8} |A|^{-5/16}\,,
	\]
	as required.

	To obtain \eqref{f:subgroup_Stepanov+'}, consider the set
	\[
	S_\tau = \{ y ~:~ r_{Q-R} (y) \ge \tau \} \,.
	\]
	Using the  assumption $|A|^{5} |QA|^{6}\le p^{8}$ and writing, for any $a\in A$,  $y= r+ qa/a$, we can set $x=qa$ and estimate the size of $S_\tau$ via Corollary \ref{t:incidences_new}  as follows:
	\begin{equation}\label{tmp:21.11_1}
	\begin{aligned} 
	\tau |S_\tau| |A| \le {I} (QA\times S_\tau , A^{-1} \times R) &  \ll  |QA|^{5/8} |S_\tau|^{1/2} |A|^{27/16} + |S_\tau|^{1/2} |A|^2 \cdot \sqrt{\max\{ 1, |QA|^2/p \}}\\
		& \ll
		|QA|^{5/8} |S_\tau|^{1/2} |A|^{27/16} \,.
		\end{aligned}
	\end{equation}
	It follows that
	\[
	|S_\tau| \ll |QA|^\frac{5}{4} |A|^\frac{11}{8} \tau^{-2} \,.
	\]
	and after summing $|S_\tau|\tau$ over a set of dyadic values of $\tau$
	\[
	\E^{+} (Q,A) \ll |QA|^\frac{5}{4} |A|^\frac{11}{8} \log |A| \,.
	\]
	%
	%
	%
	A similar argument yields  \eqref{f:subgroup_Stepanov+'_R}. 
	This completes the proof.
	$\hfill\Box$
\end{proof}

\bigskip

Let us derive a simple consequence of Corollary \ref{c:subgroup_Stepanov+}.

\begin{corollary}
	Let $\G \subseteq \F^*_p$ be a multiplicative subgroup, $|\G| \le \sqrt{p}$. 
	Then 
\begin{equation}\label{f:Q_Q1}	
	\left| \frac{\G-\G}{\G-\G} \right| \gtrsim \min \{|\G|^{2+1/18}, p^{4/3} |\G|^{-5/6}\} \,.
\end{equation}	 	
\end{corollary}
\begin{proof}
	Let $Q =\frac{\G-\G}{\G-\G}$. 
	Clearly, $Q$ enjoys $Q\G = Q$ and $Q$ contains the set $R:= \{ \frac{b-a}{c-a} ~:~ a,b,c\in \G,\, c\neq a \}$.
	We have $R=1-R$ and hence $R$ belongs to  $Q\cap (1-Q)$.   
	Using inequality \eqref{f:subgroup_Stepanov+} of Corollary \ref{c:subgroup_Stepanov+} 
	(from the proof it is easy to see that the arguments work for different sets $Q$ and $-Q$ as well), we obtain 
\[
	|Q| \gg |R|^{8/9} |\G|^{5/18} \gtrsim |\G|^{2+1/18} \,,
\]
	where we have used a well-known lower bound for $R$, namely, $|R| \gg |\G|^2/\log |\G|$ -- see, e.g. \cite{FPMS} and the references contained therein. This completes the proof.
$\hfill\Box$
\end{proof} 

\bigskip

	We now focus on  $\F = \mathbb{R}$, when it is possible to derive better estimates for the quantity $\E ({L})$, when $L$ is a grid $C\times D$.	
	To prove the next lemma we invoke\footnote{The proof of \cite[Theorem 6.1]{s_mixed} is presented in the abelian case; the general case follows immediately by replacing the  abelian  Balog--Szemer\'edi--Gowers Theorem by the non-abelian one.} \cite[Theorem 6.1]{s_mixed}.

\begin{theorem}
	Let $G$ be a group and $A\subseteq G$ a finite set,  such that $\E (A) = |A|^3/K$ and $\E_3 (A) = M|A|^4/K^2$, for some $M>0$. 
	Then there is an absolute constant $C>1$, a subset $A'\subseteq A$ and $g\in G$ such that 
\[
	|A'| \gg M^{-10} \log^{-15} M \cdot |A| \,,
\]
	and for any $k\in \mathbb N$ and arbitrary signs $\eps_j \in \{-1,1\}$, one has  
\[
	\left|\prod_{j=1}^k (g A')^{\eps_j}\right| \ll M^{Ck} K |A'| \,.
\]
\label{t:E2/E3}
\end{theorem}

\begin{lemma}
	Let $\F = \R$, $(C,D) \subseteq \F^*\times \F$, $|D|^\kappa \le |C| \le |D|^2$, $\kappa>0$,  
	and  ${L}$ be the set of affine transformations in the form $(c,d)$ or $(c, cd)$, in $\Aff(\F)$, with $c\in C,\, d\in D$.
	Then there is an absolute $\delta = \delta (\kappa) >0$ such that 
	\begin{equation}\label{f:energy_L_improved}
	\E ({L}) \lesssim  |C|^{5/2-\delta} |D|^3 \,.
	\end{equation}
\label{l:11/2-c}
\end{lemma}
\begin{proof}
	Let ${L}$ be the set of affine transformations in the form $(c,d)$,  the case $(c, cd)$ is isomorphic. 
	By Lemma \ref{l:energy_L_prod} in view of the condition $|C| \le |D|^2$, one has $\E({L}) \ll |C|^{5/2} |D|^3$. 
	
	As far as the third moment $\E_3(L)$ is concerned, it equals, with $a,a',a'', c,c', c'' \in C,\, b,b',b''$, $d,d',d''\in D$, the number of solutions of the system of equations
\[
 c/a = c'/a' = c''/a'',\, (d-b)/a = (d'-b')/a' = (d''-b'')/a'',
\]
hence
\begin{equation}\label{f:E_3(L)}
	\E_3(L)\le 
	|C| \sum_x r^3_{(D-D)/C} (x) \lesssim |C|^3 |D|^4 \,.
\end{equation}
Indeed, for any $\tau\geq 1$ one has
$$
\tau |X_\tau| :=\{x:\, r_{(D-D)/C}(x)\geq \tau\}|\;\leq \;I(D\times X_\tau, C\times D) \;\ll \;(|C||X_\tau||D|^2)^{\frac{2}{3}}\,.
$$
By the Szemer\'edi-Trotter theorem
$$
I(D\times X_\tau, C\times D) \;\ll \;(|C||X_\tau||D|^2)^{\frac{2}{3}}+|D||X_\tau| +|C||D|\,.$$
If one of the last two terms dominates the latter estimate, bound \eqref{f:E_3(L)} follows trivially; assuming the dominance of the first term leads to \eqref{f:E_3(L)} and accounts for the logarithmic term, subsumed in the $\lesssim$ symbol.

	Let $\E({L}) = |C|^{5/2} |D|^3/M$.
	Write $K=M|C|^{1/2}$, and then  $\E({L}) = |L|^3/K$. 
	Clearly, by estimate \eqref{f:E_3(L)}, the H\"older inequality and the definition of the parameter $M$, we have
	\[
		\E^2 ({L}) \le \E_3 ({L}) |{L}|^2 \lesssim |C|^{5} |D|^6 \le M^2 \E^2 ({L}) \,.
	\]
	It means that the energies $\E(L), \E_3 (L)$ are in a "critical situation"\, and we can write $\E_3 ({L}) \lesssim M^2 |L|^4 /K^2$.
	Applying  Theorem \eqref{t:E2/E3} with $K=M|C|^{1/2}$ and $M \lesssim M^2$, we find a set  
	$L_* \subseteq L$, $|L_*|\gg M^{-C_*} |L|$ and $|\prod_{j=1}^k (gL_*)^{\eps_j}| \ll M^{C_* k} |C|^{3/2} |D|$ for an absolute $C_*>1$, any positive integer $k$
	and arbitrary $\eps_j \in \{-1,1\}$.  
	Obviously, from $\E({L}) \ll |C|^{5/2} |D|^3$, it follows that 
	$|L^{-1}_* L_*| \gg_M |C|^{3/2} |D|$, so inequality $|\prod_{j=1}^k (gL_*)^{\eps_j}| \ll M^{C_* k} |C|^{3/2} |D|$ means, in particular,  that $L^{-1}_* L_*$ has small doubling (and tripling) in terms of $M$. 
	The intersection of the set  $L$ and hence the set $L_*$ with any maximal abelian subgroup of $\Aff (\R)$ is at most $\max\{ |C|, |D| \}$.
	Clearly, because of $|D|^\kappa \le |C| \le |D|^2$ the following holds $|C| = O_M (|L_*|^{2/3})$ and $|D| = O_M (|L_*|^{1-\kappa/2})$. 
	Thus, because $L_*$ does not correlate with subgroups we know by the main result of \cite{Brendan_rich} or just see Theorem \ref{t:BSzG_aff} below that 
	$(L^{-1}_* L_*)^k$ is growing
	(approximately as $O_k (|L_*|^{O(\log k)})$ if $|C|=|D|$, say). 
	It gives us a contradiction for large $k$ 
	and hence $M\gg |C|^{\delta}$ for a certain $\delta = \delta (\kappa) >0$.  
	Another way to see the same is to apply Theorem \ref{t:h-a}, which implies that the horizontal projection of 
	$L^{-1}_* L_*$ 
	has size  $O_M(1)$
	and this is nonsense because it is at least  $O_M(|C|)$. 
	This completes the proof.
$\hfill\Box$
\end{proof}

\medskip

\begin{remark} \label{r:improvement}  Lemma \ref{l:11/2-c} implies that the exponent of $|C|$ in the incidence estimate \eqref{f:incidences_new_R} improves by an absolute  $\delta(\kappa)>0$, provided that $|D|^\kappa \le |C| \le |D|^2.$ Since we do not pursue a lower bound on $\delta$, we switch from $\lesssim$ to $\ll$ bounds.
\end{remark}

\medskip 
Lemma \ref{l:11/2-c} finally enables us to prove Theorems \ref{t:9/2-c} and \ref{t:men}.

\begin{proof}[Proof of Theorem \ref{t:9/2-c}] 

	Let us consider the first case because the second is similar. 
	Without loosing of the generality assume that $0\notin A$.
	Let $L$ be the set of affine transformations in the form $(c, cd)$, $c,d\in A$. 
	Clearly, $r_{A(A+A)} (x) = \sum_{l\in {L}} A(l^{-1} x)$. 
	Then
\[
	\sigma:= \sum_{x} r^2_{A(A+A)} (x)  = \sum_x \sum_{l_1,l_2 \in L} A(l^{-1}_1 x) A(l^{-1}_2 x) \le 2\sum_{h\in \Omega} r_{L^{-1} L} (h) \sum_{a \in A} A(ha) \,,
\]
	where $\Omega = \{ h\in \Aff (\R) ~:~ \sum_{a \in A} A(ha) \ge \D \}$, $\D:= \sigma/(2|L|^2)$. 
	As in the proof of Theorem \ref{t:Elekes_new} (the case $\D \ll 1$ is trivial)
	we have, by the  Szemer\'edi--Trotter Theorem  and the Cauchy--Schwarz inequality that 
\[
	\sigma^2 \ll \E(L) \frac{|A|^4}{\D} \,.
\]
	In other words,
\[
	\sigma^3 \ll  \E(L) |A|^4 |L|^2 = \E(L) |A|^8 \,,
\]
	and applying Lemma \ref{l:11/2-c} we  obtain 
\[
	\sigma \ll |A|^{9/2-\delta/3} \,. 
\]
	This completes the proof.
$\hfill\Box$
\end{proof} 

\medskip
\begin{proof}[Proof of Theorem \ref{t:men}]

	Similarly,
	\[
	\E^{+} (A,B) |A|^2 \le | \{ sa^{-1} + b = s_1 a^{-1}_1 + b_1 ~:~ s,s_1 \in AA,\, a,a_1 \in A,\, b,b_1 \in B \} | \,.
	\]
	Using the arguments as above, we get that either trivially $\E^{+} (A,B) \ll M^2 |B|^2$, to be dropped in view of $|B| \le |A|^2$, say, or
\[
	(\E^{+} (A,B) |A|^2)^3 \ll \E(L) |A|^4 |L|^2 \,,
\] 
where $L$ is the set of affine transformations in the form $(s, b)$, $s\in AA$, $b\in B$.

	Thus, applying Lemma \ref{l:11/2-c} as well as our assumptions  $|AA| \le M|A|$,  $|A|\le  |B|^2$, we obtain, for a certain $\d >0$, that
\[
	(\E^{+} (A,B))^3 \ll_M |B|^2 \E(L) \ll_M |B|^5 |A|^{5/2-\d}
\]
as required. 
	$\hfill\Box$
\end{proof}

\bigskip
\subsection*{Appendix}


In this section we derive a group action version of the Balog--Szemer\'edi--Gowers Theorem for $G=\Aff (\F_p)$.
In \cite{Brendan_rich} a similar scheme was utilised, obtaining the forthcoming estimate (\ref{f:BSzG_aff}), which instead of the ratio $|A|/|S|$ in the right-hand side contains just $|A|$. It is owing to this potentially useful saving that we present the following.

\begin{theorem}
	Let $S \subseteq \Aff (\F_p)$ and $A\subseteq \F_p$, $|S| \le |A|$.
	Suppose, for some $M\geq 1$,  that for any $g\in \Aff (\F_p)$ and any maximal abelian subgroup $H \subset \Aff (\F_p)$ one has
	\[
	|S \cap gH| \le \frac{|S|}{M} \,.
	\]
	Then for any positive integer $k$ the following holds
	\begin{equation}\label{f:BSzG_aff}
	\sum_{x \in A} \sum_{s\in S} A(sx) \ll |S| |A| \cdot (M/ (k \log M)^2 )^{-ck2^{-k}} (|A|/|S|)^{1/(6\cdot 2^k)} \,,
	\end{equation}
	provided that 
	$p \geq |A| k^4 \log^4 M \cdot (M/ (k \log M)^2 )^{Ck}$. 
	Here $c\in (0,1), \,C\geq1$ are absolute constants.
	\label{t:BSzG_aff}
\end{theorem}

The role of the parameter $k$ in bound \eqref{f:BSzG_aff} is to make the potentially large term $|A|/|S|$ in $(|A|/|S|)^{1/(6\cdot 2^k)}$ negligible.
On the other hand, increasing  $k$ decreases the saving $(M/ (k \log M)^2 )^{-ck2^{-k}}$.

Theorem \ref{t:BSzG_aff} represents in a sense, a more technical version of Theorem \ref{t:Elekes_new}. Likewise it implies Elekes' Theorem \ref{t:Elekes}
from Introduction, by setting  $k=1$ in \eqref{f:BSzG_aff} and recalling that $|S| = |A|$.
A similar reduction was done in \cite[Theorem 28]{Brendan_rich}.

Once we have presented the proof of Theorem \ref{t:BSzG_aff} we conclude by a few lines of explanation of how it relates to the usual Balog-Szemer\'edi-Gowers theorem.

\bigskip

\begin{proof}
	One can assume that $k$ is sufficiently small, say, 
	that $k\le M^{1/2} \le |S|^{1/2}$ because otherwise \eqref{f:BSzG_aff} is trivial. 
	Now, let
	\begin{equation}\label{tmp:20.11_1}
	\sum_{x \in A} \sum_{s\in S} A(sx) = |S| |A| /K \,,
	\end{equation}
	for some $K\ge 1$.
	The task is to get a lower bound  on $K$.
	Using the Cauchy--Schwarz inequality, we get
	\[
	|A| |S|^2 K^{-2} \le \sum_{x \in A} \left( \sum_{s\in S} A(sx) \right)^2 \le \sum_{x \in A}  \sum_s r_{S S^{-1}} (s) A(sx) \,.
	\]
	Iterating, for for any $j\geq 1$, we obtain
	\[
	|A| |S|^{2^j} K^{-2^j} \le \sum_{x \in A}  \sum_s r_{(S S^{-1})^{2^{j-1}}} (s) A(sx) \,.
	\]
	Denote $\alpha_j = 2^{-1} K^{-2^j}$ and let $\Omega_j = \{ s\in G ~:~ \sum_{x \in A}  A(sx) \ge \alpha_j |A| \}$ be a popular set.
	By the pigeonhole principle
	\begin{equation}\label{tmp:19.09_1}
	2^{-1} |A| |S|^{2^j} K^{-2^j} \le \sum_{x \in A}  \sum_{s\in \Omega_j} r_{(S S^{-1})^{2^{j-1}}} (s) A(sx) \,.
	\end{equation}
	Assume for now that
	\begin{equation}\label{cond:alpha}
	\alpha_j  \gg |A|^{-1} \quad \quad \mbox{ and } \quad \quad  |A| \le 2^{-1} \alpha_j p \,.
	\end{equation}
	We will check 
	these 
	conditions later.
	
	In view of \eqref{cond:alpha}, using the incidence estimate \eqref{f:14'}, we obtain $|\Omega_j| \ll |A| \alpha_j^{-4}$.
	It follows from \eqref{tmp:19.09_1} (one can use more precise arguments, 
	but here we can manage quite roughly) that
	\[
	(|A| |S|^{2^j} K^{-2^j})^2 \ll  \sum_{s} r^2_{(S S^{-1})^{2^{j-1}}} (s) \cdot |A|^2 |\Omega_j|
	\ll \sum_{s} r^2_{(S S^{-1})^{2^{j-1}}} (s) \cdot |A|^3 \alpha_j^{-4} \,.
	\]
	Denote $$\T_{2^j}:=\sum_{s} r^2_{(S S^{-1})^{2^{j-1}}} (s),\;\;\;\mbox{ with }\;\;\T_{1} = |S|^2\,.$$
	Then the last inequality implies that
	\begin{equation}\label{tmp:19.09_2}
	\T_{2^j} \gg K^{-6\cdot 2^{j}} \frac{|S|}{|A|} \cdot |S|^{2^{j+1}-1} \,, \quad \quad j\ge 0 \,.
	\end{equation}
	Let $Q =  C K^{6 k^{-1}\cdot 2^k } (|A|/|S|)^{1/k}$ be a parameter (here $C > 1$ is a sufficiently large absolute constant)
	and suppose that for any $j \in [k]$ one has $\T_{2^j} \le |S|^{2^j} \T_{2^{j-1}} /Q$.
	One can assume that $|Q| \le |S|$ because otherwise there is nothing to prove. 
	It follows from \eqref{tmp:19.09_2} that
	\[
	|S|^{2^{k+1}-1} Q^{-k} \ge \T_{2^k} \gg K^{-6\cdot 2^{k}} \frac{|S|}{|A|} \cdot |S|^{2^{k+1}-1}\,,
	\]
	and hence there is $j\in [k]$ such that
	\[
	\T_{2^j} \ge \frac{|S|^{2^j} \T_{2^{j-1}}}{Q}  \,.
	\]
	Let $L =  C_* C k\log M$, where $C_*$ is another  sufficiently large absolute constant.
	Further, we can assume that $K \ll (M/L^2)^{ck2^{-k}} (|S|/|A|)^{1/(6\cdot 2^k)}$ because otherwise there is nothing to prove.
	Hence
	\begin{equation}\label{f:ttilda}
	K^{6\cdot 2^j} \le K^{6\cdot 2^k} \ll (M / L^2)^{6ck} \cdot |S|/|A| \,.
	\end{equation} 
	By the dyadic Dirichlet principle and the H\"older inequality there is a number $\D >0$ and a set $P = \{ s \in G ~:~ \D <  r_{(S S^{-1})^{2^{j-2}}} (s) \le 2 \D \}$, such that
	\begin{equation}\label{tmp:19.09_3}
	L^4 \D^4 \E(P) \ge \T_{2^j} \ge \frac{|S|^{2^j} \T_{2^{j-1}}}{Q} \ge \frac{(\D |P|)^2 \D^2 |P|}{Q} \,.
	\end{equation}
	Indeed, in view of \eqref{tmp:19.09_2} we can assume that
	\[
	|S|^{2^{j-1}-1} \ge \D \gg K^{-6\cdot 2^j} |S|^{2^{j-1}} |A|^{-1}
	\]
	and hence we do indeed have the upper bound \eqref{tmp:19.09_3} with the quantity $L$.
	Further from \eqref{tmp:19.09_3}, we obtain
	$\D \ge L^{-4} \T_{2^j} |S|^{-3\cdot 2^{j-1}}$ and
	\[
	\E(P) \gg L^{-4} \frac{|P|^3}{Q} \,.
	\]
	Also note that $\D^4 \E(P) \le \D^2 |P| (\D |P|)^2 \le \T_{2^{j-1}} (\D |P|)^2$ and hence from \eqref{tmp:19.09_3}, we see that
	\begin{equation}\label{tmp:19.09_4}
	\D |P| \ge \frac{|S|^{2^{j-1}}}{L^2 Q^{1/2}} \,.
	\end{equation}
	Similarly, $\D^4 \E(P) \le (\D^2 |P|) |S|^{2^j-2} |P|^2 \le \T_{2^{j-1}} |S|^{2^j-2} |P|^2$ and thus from \eqref{tmp:19.09_3}, we derive
	\begin{equation}\label{tmp:19.09_4'}
	|P| \ge \frac{|S|}{L^2 Q^{1/2}} \,.
	\end{equation}
	By the non--commutative  Balog--Szemer\'edi--Gowers Theorem, see \cite[Theorem 32]{Brendan_rich} or \cite[Proposition 2.43, Corollary 2.46]{TV}, we find $q\in P$,  $P_* \subseteq q^{-1}P$ such that
	$|P^3_*| \ll Q^{C_1} |P_*|$, $|P_*| \gg Q^{-C_1} |P|$, where $C_1 > 1$ is an absolute constant (by increasing which we can assume $P_*$ to be symmetric).
	Applying 
	Theorem \ref{t:h-a} 
	we find $x\in G$ and a maximal abelian subgroup $H$ such that
	$|P_* \cap x H| \gg |P_*| /Q^{4 C_1} \gg |P| Q^{-C_2}$ or $Q^{3C_1} |P_*| \gg p |\pi (P_*)|$.
	
	
	In the former case,
	by definition of the set $P$, the inclusion $qP_* \subseteq P$,  bound \eqref{tmp:19.09_4}
	and the definition of  $M$, we have
	\[ \begin{aligned}
	|S|^{2^{j-1}} /M
	\ge
	\sum_{s \in qx H} r_{(S S^{-1})^{2^{j-2}}} (s) & \ge \sum_{s \in qP_* \cap qx H} r_{(S S^{-1})^{2^{j-2}}} (s) > \D |qP_* \cap qx H| \\
	& \gg
	\D |P| Q^{-C_2}
	\gg
	\frac{|S|^{2^{j-1}}}{L^2 Q^{C_2+1/2}} \,.\end{aligned}
	\]
	This gives us, say,  $Q\gg (M /L^2)^{c'}$ with an absolute constant $c'>0$ (which depends just on constants in the Balog--Szemer\'edi--Gowers theorem).
	Recalling the definition of $Q$, we obtain, for a certain absolute constant $c=c(c')>0$, that  
	\begin{equation}\label{tmp:19.09_5}
	K \gg (M / L^2)^{ck2^{-k}} (|S|/|A|)^{1/(6\cdot 2^k)}\,,
	\end{equation}
	and (\ref{f:BSzG_aff}) follows.
	
	
	Otherwise,  let $Q^{3C_1} |P_*| \gg p |\pi (P_*)|$.
	Then $Q^{3C_1} |P| \gg p$ and hence in view of \eqref{tmp:19.09_2}, we get
	\[
	K^{-6\cdot 2^{j}}  |S|^{2^{j+1}} |A|^{-1} \ll \T_{2^j} \le L^4 \D^4 |P|^3 \le L^4 (\D |P|)^4/|P| \ll L^4 Q^{3C_1} |S|^{2^{j+1}} /p \,.
	\]
	Again, we can assume that $Q\ll (M/L^2)^{c'}$, for otherwise there is nothing to prove, and the latter inequality, in view of \eqref{f:ttilda} yields
	\[
	p \ll |A| L^4 (M /L^2)^{c''} K^{6\cdot 2^k} \ll |S| \cdot L^4 (M/L^2)^{Ck}\,.
	\le |A| \cdot L^4 (M/L^2)^{Ck} 
	\]
	This contradicts our assumption on $A$.

	It remains to check the conditions \eqref{cond:alpha}.
	The first inequality therein being violated implies a better estimate on $K$ than \eqref{f:BSzG_aff}.
	If the second bound in \eqref{cond:alpha} fails, then $4|A| K^{2^j} \ge p$.
	This, combined with the converse of bound \eqref{tmp:19.09_5} means that $|A| \ge |A| \cdot (|S|/|A|)^{1/6} \gg p (M/L^2)^{-Ck}$, which yields   a contradiction with the assumption on $p$ in the statement of the theorem.
	This completes the proof.
	$\hfill\Box$
\end{proof}

\medskip
It is easy to see that the above argument 
yields a 
relatively 
short
proof of the usual abelian\, Balog--Szemer\'edi--Gowers Theorem as well as a proof for groups  $\Aff (\mathbb{C})$ and $\Aff (\mathbb{R})$
(in the latter case the conditions involving $p$ obviously get dropped).
Indeed, if the common energy of two subsets $A,B$ of an abelian group is large, namely, $\E^{+} (A,B) \ge |A| |B|^2/M$, then 
\begin{equation}\label{f:BszG_ab}
	\sum_{b \in B} \sum_{a\in A} P(a+b)  \ge \frac{|A| |B|}{2M} \,, 
\end{equation}
where $P$ is the set of all $x$ such that the equation $a+b = x$, $a\in A$, $b\in B$ has at least  $|B|/(2M)$ solutions.
Clearly, $|A|/(2M) \le |P| \le 2M |A|$ and hence size of $P$ is comparable    with $|A|$ (one thinks of $M$ as a rather small power of 
$|B|$). 
Equation \eqref{f:BszG_ab} has the same form as \eqref{tmp:20.11_1}, representing the action of affine transformations $x\to x+ b$, $b\in B$ that belong to  the unipotent subgroup $U_0$: this is a simplification of the scenario considered throughout the affine group part of this paper.

\section*{Acknowledgement} The authors are extremely grateful to Oliver Roche-Newton and Audie Warren for having followed through many arguments in this paper, and whose critical eyes had helped them spot a few errors in its earlier version. Special thanks to Harald Helfgott for pointing out the quantitative saving in Corollary \ref{c:h}.  More thanks to Oliver Roche-Newton for pointing out the non-symmetric (that is for $|C|\neq |D|$) case of Lemma \ref{l:11/2-c} and the scope of its applications.

\bigskip
\noindent{Misha Rudnev\\
	School of Mathematics,\\
	University Walk, Bristol BS8 1TW, UK\\
	{\tt misarudnev@gmail.com}

\bigskip
\noindent{Ilya D.~Shkredov\\
	Steklov Mathematical Institute,\\
	ul. Gubkina, 8, Moscow, Russia, 119991}
\\
and
\\
IITP RAS,  \\
Bolshoy Karetny per. 19, Moscow, Russia, 127994\\
and
\\
MIPT, \\
Institutskii per. 9, Dolgoprudnii, Russia, 141701\\
{\tt ilya.shkredov@gmail.com}

\end{document}